\documentclass{amsart}

\usepackage{amssymb}
\usepackage[all]{xy}
\usepackage{caption}
\usepackage{subfig}
\usepackage[pdftex]{hyperref}
\usepackage{graphicx}
\usepackage{array}
\usepackage{physics}
\usepackage{color}
\usepackage{stmaryrd}
\usepackage{bm}
\usepackage{tikz}
\usetikzlibrary{matrix,arrows}
\usepackage{comment}

\newtheorem{thm}{Theorem}
\newtheorem*{thm*}{Theorem}

\newtheorem{lem}[thm]{Lemma}
\newtheorem{prop}[thm]{Proposition}
\newtheorem{cor}[thm]{Corollary}
\theoremstyle{definition}

\newtheorem{rem}[thm]{Remark}

\newcommand{\CC}{{\mathbb C}}
\newcommand{\ZZ}{{\mathbb Z}}
\newcommand{\RR}{{\mathbb R}}
\newcommand{\cX}{{\mathcal X}}
\newcommand{\cY}{{\mathcal Y}}
\newcommand{\G}{{\Gamma}}
\newcommand{\la}{\langle}
\newcommand{\ra}{\rangle}
\newcommand{\g}{\gamma}
\newcommand{\x}{\times}

\DeclareMathOperator{\Spec}{Spec}
\DeclareMathOperator{\id}{Id}
\DeclareMathOperator{\Id}{Id}

\DeclareMathOperator{\GL}{GL}
\DeclareMathOperator{\SU}{SU}
\DeclareMathOperator{\U}{U}

\DeclareMathOperator{\SL}{SL}

\DeclareMathOperator{\Sym}{Sym}    
\DeclareMathOperator{\diag}{diag}

\title{Geometry of $\SU(3)$-character varieties of torus knots}

\author{\'Angel Gonz\'alez-Prieto}

\address{Departamento de \'Algebra, Geometr\'ia y Topolog\'ia, Facultad de Ciencias Matem\'aticas, Universidad Complutense de Madrid, Plaza Ciencias 3, 28040 Madrid Spain.}
\address{Instituto de Ciencias Matem\'aticas (CSIC-UAM-UCM-UC3M), C.\ Nicolás Cabrera 13-15, 28049 Madrid Spain.}\email{angelgonzalezprieto@ucm.es}

\author{Javier Mart\'inez}

\address{Departamento de Matem\'atica Aplicada, Ciencia e Ingeniería de los Materiales y Tecnología Electrónica. E.S. Ciencias Experimentales y Tecnología, Universidad Rey Juan Carlos, C.\ Tulipán 0, 28933 Móstoles, Madrid Spain.}
\email{javier.martinezmar@urjc.es}

\author{Vicente Mu\~noz}

\address{Departamento de \'Algebra, Geometr\'ia y Topolog\'ia, Facultad de Ciencias Matem\'aticas, Universidad Complutense de Madrid, Plaza Ciencias 3, 28040 Madrid Spain.}\email{vicente.munoz@ucm.es}

\begin{document}

\begin{abstract}
We describe the geometry of the character variety of representations of the knot group $\G_{m,n}=\la x,y| x^n=y^m\ra$
into the group
$\SU(3)$, by stratifying the character variety into strata correspoding to totally reducible representations,
representations decomposing into a $2$-dimensional and a $1$-dimensional representation, and 
irreducible representations, the latter of two types depending on whether the matrices have 
distinct eigenvalues, or one of the matrices has one eigenvalue of multiplicity $2$. We describe how
the closure of each stratum meets lower strata, and use this to compute the compactly supported
Euler characteristic, and to prove that the inclusion of the character variety for $\SU(3)$ into the
character variety for $\SL (3,\CC)$ is a homotopy equivalence. 
\end{abstract}

\maketitle

\begin{center}
{\center \em Dedicated to Prof.\ Jos\'e  Manuel Rodr\'{\i}guez Sanjurjo on the ocassion of his 70th birthday.}
\end{center}

\section{Introduction}\label{sec:introduction}

Let $\G$ be a finitely generated group a $G$ a real or complex algebraic group. A
\textit{representation} of $\G$ into $G$ is a group homomorphism $\rho: \G\to G$.
Consider a presentation $\G=\la \g_1,\ldots, \g_k | \{r_{\lambda}\}_{\lambda \in \Lambda} \ra$, 
where $\Lambda$ is the (possibly infinite) indexing set of relations of $\G$. The map $\rho$ is completely
determined by the $k$-tuple $(A_1,\ldots, A_k)=(\rho(\g_1),\ldots, \rho(\g_k))$
subject to the relations $r_\lambda(A_1,\ldots, A_k)=\id$, for all $\lambda \in \Lambda$. 

In this way, the set of representations of $\G$ into $G$ is in bijection with the set
 \begin{align}\label{eq:rep-var}
 \begin{split}
R(\G, G) &=\, \{(A_1,\ldots, A_k) \in G^k \, | \,
 r_\lambda(A_1,\ldots, A_k)=\id ,\,\, \textrm{for all } \lambda \in \Lambda\, \}\subset G^{k}\, .
 \end{split}
 \end{align}
Notice that when $G$ is an affine algebraic group, even if the set of relations $\Lambda$ is infinite,
the set $R(\G, G)$ is defined by finitely many equations thanks to the noetherianity of the coordinate ring of $G^k$. 

In this setting, two representations $\rho$ and $\rho'$ are said to be
equivalent if there exists $g \in G$ such that $\rho'(\gamma)=g^{-1} \rho(\gamma) g$,
for every $\gamma \in \G$. In the case that there is a faithful representation $G \subset \GL(V)$, this means that $\rho$ and $\rho'$ are the same representation up to a $G$-change of basis of $V$. The moduli space of representations can be thus obtained as the GIT quotient
 $$
 X(\G, G) = R(\G, G) \hspace{-1pt}\sslash\hspace{-1pt} G\, .
 $$
Recall that, by definition, the GIT quotient of the affine variety $R(\G, G) = \Spec(A)$ under the action of a reductive group $G$ is 
$R(\G, G) \hspace{-1pt}\sslash\hspace{-1pt} G = \Spec A^G$, where $A^{G}$ is the (finitely generated) $k$-algebra of $G$-invariant elements of $A$.

An important instance is $G = \GL(r, \CC)$, in which we recover the classical notion of a linear representation of $\G$ as a $\G$-module structure on the vector space $\CC^r$. This case has been thoroughly studied in \cite{Hausel-Rodriguez-Villegas:2008}, among others. It is worth noticing that, when $\G$ is actually a finite group, the vector space $\CC^r$ can be equipped with a $\G$-invariant hermitian metric. Hence, any representation $\rho: \G \to \GL(r, \CC)$ descends to an $\U(r)$-representation
	\[
\begin{displaystyle}
   \xymatrix
   {	\G \ar[rr]^{\rho} \ar@{-->}[rrd]_{\tilde{\rho}} && \GL(r, \CC) \\
    && \U(r) \ar@{^{(}->}[u]
      }
\end{displaystyle}   
\]
In other words, this means that $X(\G, \GL(r, \CC)) = X(\G, \U(r))$. However, in the general case in which $\G$ is only finitely generated, such an invariant metric may not exist so $X(\G, \U(r))$ is only a real subvariety of $X(\G, \GL(r, \CC))$. Similar considerations can be done in the case in which we fix the determinant of the representation, so we analyze the descending property of representations induced by the inclusion $\SU(r) \hookrightarrow \SL(r. \CC)$, which exhibits $X(\G, \SU(r))$ as a real subvariety of $X(\G, \SL(r, \CC))$.

The properties of these subvarieties have been widely studied in the literature, as in \cite{acosta2019} or \cite{casimiro2014}. 
Furthermore, in \cite{florentino2020flawed} (see also \cite{bergeron2015topology,florentino2009topology,florentino2014topology}) the authors proved that, when $\G$ is a free product of nilpotent groups or a star-shaped RAAG (Right Angled Artin Groups), 
the inclusion $X(\G, \U(r)) \hookrightarrow X(\G, \SL(r, \CC))$ is a deformation retract for any reductive group $G$ (a property called flawed). On the contrary, for $\G = \pi_1(\Sigma_g)$, the fundamental group of a compact orientable surface $\Sigma_g$
of genus $g \geq 2$, 
this inclusion is never a homotopy equivalence when $G$ is reductive and non-abelian (it is said that $\G$ is a flawless group in the language of  \cite{florentino2020flawed}). These character varieties have been extensively studied, as in \cite{ho2018textit} for $g =2$ and $G = \SU(2)$ or in \cite{goldman2021mapping} regarding the ergodic properties of the action of the mapping class group of $\Sigma_g$.

However, the case in which $\G = \pi_1(S^3 - K)$ is the fundamental group of the $3$-dimensional complement of a knot $K \subset S^3$ is not fully understood. The analyses of \cite{martinezmun2015su2} and \cite{munoz2009sl} prove that such inclusion is a deformation retract in the case $\SU(2) \hookrightarrow \SL(2, \CC)$ and $\G$ the torus knot group. However, despite the geometry of the $\SL(3, \CC)$-character varieties of torus knots has been studied in \cite{munozporti2016} and the $\SL(4, \CC)$-character varieties in \cite{gonzalez2020motive}, almost nothing in known for the compact counterparts $\SU(r)$ for $r \geq 3$.

The goal of this paper is to give the first steps towards this aim and to study the geometry of the $\SU(r)$-character varieties of torus knots for $r \geq 3$, with particular attention to the case $r = 3$. In Section \ref{sec:rep-var} we discuss some generalities about these character varieties, proving in particular that $\SU(r)$-representations are always semi-simple.

This property allows us to stratify the $\SU(r)$-character variety in a natural way depending on the dimensions of the irreducible pieces appearing in its semi-simple decomposition, as discussed in Section \ref{sec:stratification}. Two strata can be fully understood in this setting: the stratum of totally reducible representations (Section \ref{sec:totally-reducible}) and the one of irreducible representations (Section \ref{sec:irreducible}). In the former case, we show that the totally irreducible representations form a fiber bundle over $S^1$ with fiber a certain $(r-1)$-dimensional simplex. In the latter, we characterize such irreducible representations as a particular subspace of the bi-orbit space $(S^1)^r \backslash \U(r) \,\slash (S^1)^r$.

This description is exploited in Section \ref{sec:u2-char} to describe the geometry of the $\U(2)$-character variety of torus knots. However, the core of the work starts at Section \ref{sec:SU3char}. There, we provide explicit expression of each of the strata that comprise the $\SU(3)$-character variety. To be precise, we obtain the following result.

\begin{thm*}
Let $\Gamma$ be the $(n,m)$-torus knot group with $n$ and $m$ coprime. The $\SU(3)$-character variety of $\G$ decomposes into strata
$$
	X(\G, \SU(3)) = \cY_{(1,1,1)} \sqcup \cY_{(3)} \sqcup \cY_{(2,1)},
$$
corresponding to totally reducible representations, irreducible representations and representations that decompose into the direct sum of a  $2$-dimensional representation and a $1$-dimensional representation, respectively. The closure of each statum is as follows:
\begin{itemize}
	\item $\cY_{(1,1,1)}$ is already a closed $2$-dimensional triangle.
	\item $\overline{\cY}_{(3)}$ is the disjoint union of two spaces. The first one is made of $\frac{1}{2}(n-1)(m-1)(n+m-4)$ disjoint closed $2$-dimensional triangles. The second one is a disjoint union of $\frac{1}{12}(n-1)(n-2)(m-1)(m-2)$ spaces which fiber over a closed $2$-dimensional triangle, with fiber $S^2$ over the interior. The fibration over the boundary of the triangle is isomorphic to a closed M\"obius band.
	\item  $\overline{\cY}_{(2,1)}$ is a disjoint union of $\lfloor\frac12(n-1)\rfloor \frac12(m-1)$ closed cylinders and, when $n$ is even, $\frac12(m-1)$ closed M\"obius bands.
\end{itemize}
\end{thm*}

Furthermore, in Section \ref{sec:intersection-tr} we provide an explicit description of how these strata intersect to form the CW-complex structure of $X(\G, \SU(3))$. From this information it is possible to extract some homological invariants, as done in Section \ref{sec:homological-invariants}, which give rise to the following result.

 \begin{thm*}
Let $\Gamma$ be the $(n,m)$-torus knot group with $n$ and $m$ coprime. The $\SU(3)$-character variety of $\G$ satisfies the following:
\begin{enumerate}
	\item For $n,m$ odd, the Euler characteristic with compact support of $X(\G, \SU(3))$ is
	$$
		\chi_c(X(\G, \SU(3))) = 1 + (n-1)(m-1) \left( \frac{n+ m - 4}{2} + \frac{5(n-2)(m-2)}{12} \right).
	$$
	\item For $n = 2$ and odd $m$, the homology of $X(\G, \SU(3))$ is
	$$
		H_k(X(\G, \SU(3))) = \left\{ \begin{array}{ll} 
		\ZZ & \textrm{if } k = 0, \\ 
		0 & \textrm{if } k = 1, \\ 
		\ZZ^{\frac{1}{2}(m-1)(m-2)} & 
		\textrm{if } k = 2, \\ 0  & \textrm{if } k \geq 3.
		\end{array} \right.
	$$
\end{enumerate}
\end{thm*}

From this result and the ones of \cite{munozporti2016}, we get that the homologies of the $\SU(3)$ 
and $\SL(3,\CC)$-character varieties agree for $n=2$, $m$ odd.
Moreover, we have the following result.

 \begin{thm*}
Let $\Gamma$ be the $(2,m)$-torus knot group with $m$ odd. The inclusion
$$
	X(\G, \SU(3)) \hookrightarrow X(\G, \SL(3, \CC))
$$
is a homotopy equivalence.
\end{thm*}

This is an indication that knot groups $\G=\G_{m,n}$ are probably flawless, that is 
the inclusion $X(\G, K) \hookrightarrow X(\G, G)$ of the associated character variety 
to the maximal compact subgroup $K \subset G$ is a deformation retract, for 
any $G$.

\subsection*{Acknowledgements}

The first author has been partially supported by Comunidad de Madrid R+D Project PID/27-29 and Ministerio de Ciencia e Innovaci\'on Project PID2021-124440NB-I00 (Spain), and the third author has been partially supported by Ministerio de Ciencia e Innovaci\'on Project PID2020-118452GB-I00 (Spain).

\section{$\SU(r)$-representation varieties of torus knots}\label{sec:rep-var}

In this section, we discuss some properties of the $\U(r)$ and 
$\SU(r)$-character varieties of torus knots. Recall that, given coprime natural numbers $n, m$, the $(n,m)$-torus knot group is the group
$$
	\G_{n,m} = \langle a,b \,\left| \, a^n = b^m \right. \rangle.
$$
This group arises in low dimensional topology as the fundamental group of the knot complement $\RR^3-K$, where $K$ is the so-called $(n,m)$-torus knot that gives $n$ turns around the meridian and $m$ turns around the parallel of the naturally embedded torus $S^1 \times S^1 \subset \RR^3$. In other words, $K$ is the image in $\RR^3$ of the skew line in the square representation of the torus with rational slope $n/m$.

In this way, the associated $G$-representation variety is
$$
	R(\G_{n,m}, G) = \{(A,B) \in G \left| \, A^n = B^m \right. \}.
$$
These varieties have been previously studied in \cite{gonzalez2020motive,martinezmun2015su2,munoz2009sl,munozporti2016}, among others.

If $G = \GL(r, \CC), \SL(r, \CC), \U(r)$ or $\SU(r)$, a representation $\rho$ is \textit{reducible} if there exists some proper linear
subspace $W\subset \CC^r$ such that for all $\g \in \G$ we have 
$\rho(\g)(W)\subset W$; otherwise $\rho$ is
\textit{irreducible}. 
If $\rho$ is reducible, then there is a flag of subspaces $0=W_0\subsetneq W_1\subsetneq \ldots \subsetneq W_r=\CC^r$
such that $\rho$ leaves $W_i$ invariant, and it induces an irreducible representation $\rho_i$ in the quotient
$V_i=W_i/W_{i-1}$, $i=1,\ldots,r$. Then $\rho$ and $\hat \rho=\bigoplus \rho_i$
define the same point in the quotient $X(\G,G)$. We say that $\hat\rho$ is a semi-simple
representation, and that $\rho$ and $\hat\rho$ are 
S-equivalent. The space $X(\G,G)$ parametrizes semi-simple representations
\cite[Thm.~ 1.28]{lubotzky1985}. 

Now, let us suppose that $G = \U(r)$ and let $(A,B) \in R(\G_{n,m}, G)$. In that case, both $A$ and $B$ are diagonalizable in an orthonormal basis. 
So there exists $Q \in \U(r)$ such that $Q^{-1}AQ$ is diagonal, and the representation $(A,B)$ is equivalent to a representation of the form
\begin{equation}\label{eq:standard-form}
	\left(A=\begin{pmatrix}\lambda_1 & 0 & \ldots & 0 \\ 0 & \lambda_2 & \ldots & 0 \\ \vdots &  & \ddots & \vdots \\ 0 & 0 &\ldots &   
	\lambda_r \end{pmatrix},  
	B=\begin{pmatrix}b_{11} & b_{12} & \ldots & b_{1r} \\ b_{21} & b_{22} & \ldots & b_{2r} \\ \vdots &  & \ddots & \vdots \\ b_{r1} & b_{r2} &
	\ldots &  b_{rr} \end{pmatrix} \right).
\end{equation}
Here, the eigenvalues $\lambda_1, \ldots, \lambda_r \in S^1$ and the column vectors of $B$, namely
$b_1 = (b_{11}, b_{21}, \ldots, b_{r1}), \ldots$, $b_r = (b_{1r}, b_{2r}, \ldots, b_{rr})$ form an orthonormal basis of 
$\CC^r$. Note that this standard form of the representation is not unique. 

In this way, the representation $(A,B)$ is reducible if and only if it has a standard form (\ref{eq:standard-form}) such that all $b_{ij} = 0$ for $s < i \leq r$ and $1 \leq j \leq r-s$ for some $1 < s < r$, i.e.\ if the $s \times (r-s)$ left-bottom block of $B$ vanishes. However, the fact that $A, B \in \U(r)$ allows us to get a stronger result.

\begin{prop}\label{prop:semi-simple}
Every $\U(r)$-representation is semi-simple.
\end{prop}

\begin{proof}
Let $(A,B)$ be the standard form (\ref{eq:standard-form}) of the representation with its $s \times (r-s)$ left-bottom block being zero. This means that $\langle b_1, \ldots, b_{r-s}\rangle = \langle e_1, \ldots, e_{r-s}\rangle$. Hence, since $b_{j} \in \langle b_1, \ldots, b_{r-s}\rangle^\perp =  \langle e_1, \ldots, e_{r-s}\rangle^{\perp}$ for $j > r-s$, we have that $b_{ij} = 0$ for $i \leq r-s$. Therefore, also the $(r-s) \times s$ right-upper block vanishes and thus $(A,B)$ is semi-simple.
\end{proof}

\begin{rem}
An alternative proof of Proposition \ref{prop:semi-simple} is the following. Recall that the $\U(r)$-orbit of any representation contains in its closure a semi-simple representation. As $\U(r)$ is a compact group, its orbits are closed and thus the limit semi-simple representation actually lies in the orbit.
\end{rem}

We can diagonalize $B$ and write it as
 \begin{equation}\label{eq:standard-form-eigen}
  B = P\begin{pmatrix}\nu_1 & 0 & \ldots & 0 \\ 0 & \nu_2 & \ldots & 0 \\ \vdots &  & \ddots & \vdots \\ 0 & 0& \ldots & \nu_r \end{pmatrix}P^{-1}, 
  \qquad 
  P = \begin{pmatrix}p_{11} & p_{12} & \ldots & p_{1r} \\ p_{21} & p_{22} & \ldots & p_{2r} \\ \vdots &  & \ddots & \vdots \\ 
  p_{r1} & p_{r2} &\ldots  & p_{rr} \end{pmatrix},
\end{equation}
for a matrix $P$ whose column vectors $p_1 = (p_{11}, p_{21}, \ldots, p_{r1}), \ldots , p_r = (p_{1r}, p_{2r}, \ldots, p_{rr})$ form an 
orthonormal basis of $\CC^r$, namely an orthonormal basis of eigenvectors of $B$,  and some eigenvalues $\nu_1, \ldots, \nu_r \in S^1$. Observe that, by Proposition \ref{prop:semi-simple}, since a reducible representation is semi-simple, if $(A,B)$ is reducible then there must exist a choice of eigenvectors that gives a block structure in $P$. In other words, the representation $(A,B)$ is irreducible if and only if there exists no vanishing sub-minor in $P$
in \emph{any} possible expression as in (\ref{eq:standard-form-eigen}).

\begin{rem}
Completely analogous descriptions can be done in the $\SU(r)$ case. In this setting, any representation can also be put in the forms (\ref{eq:standard-form}) and (\ref{eq:standard-form-eigen}) but the eigevalues must additionally satisfy $\prod \lambda_i = 1$ and $\det(B) = 1$ (equivalently, $\prod \nu_i = 1$). Moreover, since the $\U(r)$-orbits are the same as the $\SU(r)$-orbits, Proposition \ref{prop:semi-simple} also holds for $\SU(r)$-representations.
\end{rem}

The aim of this paper is to compare the $\GL(r, \CC)$ and $\SL(r, \CC)$-character varieties with their counterparts for their maximal compact subgroups $\U(r)$ and $\SU(r)$. To shorten notation, we shall denote
$$
\begin{array}{lll}
\mathcal{X}_{r}  =X(\Gamma_{m,n},\SL(r,\mathbb{C})), &&  \tilde{\mathcal{X}}_{r}=X(\Gamma_{m,n},\GL(r,\mathbb{C})), \\
\mathcal{Y}_{r}  =X(\Gamma_{m,n},\SU(r)),  && \tilde{\mathcal{Y}}_{r}=X(\Gamma_{m,n},\U(r)).
\end{array}
$$

Observe that the eigenvalues of $A$ and $B$ induce maps
\begin{equation}\label{eq:eigen-maps}
	H: \cY_r \to \Sym^r(S^1) \times \Sym^r(S^1), \qquad \tilde{H}: \tilde{\cY}_r \to \Sym^r(S^1) \times \Sym^r(S^1).
\end{equation}
Here, $\Sym^r(X)$ denotes the symmetric product $\Sym^r(X) = X^r/{\mathfrak S}_r$, where the symmetric group ${\mathfrak S}_r$ 
acts by permutation of the factors. These maps assign $\tilde{H}(A,B) = (\{\lambda_1, \ldots, \lambda_r\}, \{\nu_1, \ldots, \nu_r\})$ and analogously for $H$. 
However, it is not a fibration and, indeed, $\tilde{H}$ is not even surjective since, as $A^n = B^m$, we must have that $\{\lambda_1^n, \ldots, \lambda_r^n\} = \{\nu_1^m, \ldots, \nu_r^m\}$. However, we can stratify the base space $\Sym^r(S^1) \times \Sym^r(S^1)$ to get control of the maps $H$ and $\tilde{H}$.

\section{Stratification of the character variety}\label{sec:stratification}

Consider a partition $\pi=(r_1,\overset{(a_1)}{\ldots},r_1,\ldots,r_s,\overset{(a_s)}{\ldots},r_s)$ of $r$, that is, $r=a_1r_1+\ldots+a_sr_s$, where $r_1>r_2>\ldots>r_s>0$ and $a_i\geq 1$. 
We can consider the locally closed subvariety $\tilde{Y}_\pi \subset \tilde{\cY}_r$ of isomorphism classes of representations of the form
\begin{equation}
\rho=\bigoplus_{t=1}^{s}\bigoplus_{l=1}^{a_l} \rho_{tl},  \label{subrep}
\end{equation}
where $\rho_{tl}: \Gamma \rightarrow \U(r_t)$ is an irreducible representation. In particular, for $\pi=(1,\overset{(r)}{\ldots},1)$, $\tilde{Y}_{\pi}$ corresponds to the totally reducible representations, which we shall denote by $\tilde{\mathcal{Y}}^{\textup{TR}}_r$; whereas for $\pi=(r)$, $\tilde{Y}_{\pi}$ corresponds to the irreducible representations, which we shall denote by $\tilde{\mathcal{Y}}^{\ast}_r$.

By its very definition, for any partition $\pi=(r_1,\overset{(a_1)}{\ldots},r_1,\ldots,r_s,\overset{(a_s)}{\ldots},a_s)$ of $r$, we have an isomorphism
\begin{equation}\label{eq:partition-product}
	\tilde{Y}_{\pi}=\prod_{t=1}^{s} \Sym^{a_t} \tilde{\mathcal{Y}}^{\ast}_{r_t}.
\end{equation}
Moreover, by Proposition \ref{prop:semi-simple}, every $\U(r)$-representation is semi-simple, so it can be written in the form (\ref{subrep}) for some partition $\pi$. Hence, we have a natural stratification
$$
\tilde{\mathcal{Y}}_{r}=\bigsqcup_{\pi\in\Pi_r} \tilde{Y}_{\pi} = \bigsqcup_{\pi\in\Pi_r} \prod_{t=1}^{s} \Sym^{a_t} \tilde{\mathcal{Y}}^{\ast}_{r_t},
$$
where $\Pi_r$ is the set of all partitions of $r$. 

A similar decomposition can be set for $\SU(r)$,
$$
\mathcal{Y}_{r}=\bigsqcup_{\pi\in\Pi_r} Y_{\pi},
$$
where $Y_\pi=\tilde{Y}_\pi\cap\mathcal{Y}_r$ are those representations from (\ref{subrep}) with $\prod_{t,l} \det(\rho_{tl})=1$. However, observe that in this setting we no longer have an analogous decomposition (\ref{eq:partition-product}).

As in the $\U(r)$-case, for $\pi=(1,\overset{(r)}{\ldots},1)$, we get $Y_{\pi} = {\mathcal{Y}}^{\textup{TR}}_r$,
the set of totally reducible representations, and for $\pi=(r)$, we get ${Y}_{\pi} = {\mathcal{Y}}^{\ast}_r$, the set of irreducible representations.


\subsection{The totally reducible locus}\label{sec:totally-reducible}

In this section, we shall study the stratum of totally reducible representations, corresponding to the partition $\pi=(1,\overset{(r)}{\ldots},1)$. As we shall show, this space is strongly related to symmetric products of circles.

\begin{lem}\label{lem:irred-dim1}
$\tilde{\mathcal{Y}}_1 = \tilde{\mathcal{Y}}_1^\ast \cong S^1$.
\end{lem}

\begin{proof}
Given $(\lambda,\nu)\in \tilde{\mathcal{X}}_1$, if $\lambda^n =\nu^n$ and $m,n$ are coprime, there exists a unique $t\in \mathbb{C}^{\ast}$ such that $\lambda=t^m, \nu=t^n$, so that $\tilde{\mathcal{X}}_1\cong \mathbb{C}^\ast$. If $\abs{\lambda}=\abs{\nu}=1$, 
then $\abs{t}=1$ and we get that $\tilde{\mathcal{Y}}_1 \cong S^1 \subset \tilde{\mathcal{X}}_1 \cong \mathbb{C}^{\ast}$.
\end{proof}

\begin{prop}\label{prop:totally-reducible-Ur}
We have an isomorphism
$$
	\tilde{\mathcal{Y}}^{\textup{TR}}_r \cong \Sym^r(S^1),
$$
where $\Sym^{r}(S^1)$ is a fiber bundle $\Sym^{r}(S^1) \to S^1$ with fiber the $(r-1)$-simplex
$$
	\Delta_{r-1} =
	 \left\{ (u_1,\ldots,u_{r})\in \mathbb{R}^{r} \, \big| \, u_i\geq 0, \quad \sum_{i=1}^{r} u_i = 1\right\},
%
$$
and monodromy given by the map $(u_1,\ldots,u_{r}) \mapsto (u_{r}, u_1,\ldots, u_{r-1})$.
\end{prop}

\begin{proof}
We have $\tilde{\mathcal{Y}}_1\cong S^1$ by Lemma \ref{lem:irred-dim1} and  $\tilde{\mathcal{Y}}^{\textup{TR}}_r \cong \Sym^r(S^1)$ by (\ref{eq:partition-product}), where a totally reducible representation is given by a pair $(A,B)$, 
where $A=\diag(t_1^m,\ldots,t_r^m)$ and $B=\diag(t_1^n,\ldots,t_r^n)$, $\abs{t_i}=1$, 
and the isomorphism is given by $(A,B)\mapsto \lbrace t_1,\ldots,t_n\rbrace \in \Sym^r(S^1)$. In other words, there is a unique unitary diagonal matrix $D=\diag(t_1,\ldots,t_r)$ for any totally reducible pair $(A,B)$ that is simultaneously an $m$-th root for $A$ and an $n$-th root for $B$.

In order to describe $\Sym^r(S^1)$, we can consider the map $p:\Sym^r(S^1) \to S^1$ given by the determinant of $D$, 
$p(\lbrace t_1,\ldots,t_r \rbrace )=t_1\cdots t_r$. When the pair is a $\SU(r)$-representation so that $(A,B)\in {\mathcal{Y}}^{\textup{TR}}_r$, 
$\det(t_1\cdots t_r)=1$ so that $p^{-1}(1)={\mathcal{Y}}^{\textup{TR}}_r\subset \tilde{\mathcal{Y}}^{\textup{TR}}_r$.

We check now that ${\mathcal{Y}}^{\textup{TR}}_r$, the fiber over $1$, is an $(r-1)$-simplex.
Given $\lbrace t_1,\ldots,t_r\rbrace \in \Sym^r(S^1)$ such that $\prod t_i=1$, we can regard them as $n$ unordered points on the circle and choose an initial point $t_1$. After this choice, we can order the remaining points anticlockwise and also choose liftings $s_1,\ldots, s_r \in \mathbb{R}$ so that, via the exponential map, $e^{2\pi is_i}=t_i$ and $s_1\leq s_2 \leq 
\ldots \leq s_r$, using the logarithm with branch cut at $t_1$. 
If the multiplicity of the first point is greater than one, so that some of them are lifted in last position having gone once around the circle, the only consequence is that $s_r\leq s_1+1$.

 Since $\prod t_i =1$, we obtain that $\sum s_i\in \mathbb{Z}$. In fact, any integer can be obtained as $\sum s_i$, choosing appropiately the first point (choosing the second point as the first and placing the first as last increases the sum by 1, if we keep the branch cut fixed) and the logarithm (the total sum changes by multiples of $r$), so we can assume that $\sum s_i=0$. This determines uniquely the first point in the cyclic order, and hence uniquely
 determines the unordered set of points. Therefore
 the fiber $p^{-1}(1)$ is bijective to the set:
\begin{equation} \label{eq:redsimplex}
 A:= \lbrace (s_1,\ldots,s_r)\in \mathbb{R}^r \mid \sum s_i =0,  s_i\leq s_{i+1}, s_r \leq s_1+1 \rbrace.
\end{equation}
Define the map $\varphi:A \to Y_r^{\textup{TR}}$, given by 
 $$
  \varphi(s_1,\ldots,s_r)= \{ e^{2\pi i s_1}, \ldots, e^{2\pi i s_r} \} \in (S^1)^r/{\mathfrak S}_r\, .
  $$
This map is continuous, the first space is compact, and the second is Hausdorff, hence it is an homeomorphism with its image,
$\varphi(A)\cong Y_r^{\textup{TR}}$. 
Now take  $u_i=s_{i+1}-s_i$, $i=1,\ldots, r-1$, $u_r=s_1+1-s_r$. Therefore, we have the conditions $u_i \geq 0$ and 
$s_i=s_1+u_1+\ldots+u_{i-1}$, $i=1,\ldots, r$. The sum is 
 $0=\sum\limits_{i=1}^r s_i = r s_1+ (r-1)u_1+\ldots +u_{r-1}$, which uniquely determines $s_1= -\sum\limits_{i=1}^{r-1} \frac{r-i}r u_i$.

To check that $p$ defines a fibre bundle structure, note that $\Sym^r(S^1)$ can be trivialized over any $t\in S^1$ via the map 
$\phi: p^{-1}(t)\times (S^1- \lbrace -t \rbrace) \longrightarrow p^{-1}(S^1 -\lbrace -t \rbrace)$,
 $$
 \phi (\lbrace t_1,\ldots, t_r \rbrace, te^{i\theta}) = \lbrace e^{i\theta/r}t_1, \ldots, e^{i\theta/r}t_r \rbrace,
 $$
for $\theta\in (-\pi,\pi)$. In other words, the fibre over $1$ and the fibre over $t\in S^1$ can be identified by multiplying by a suitable $r$-th root of $\diag(t,\ldots,t)$, well defined after taking a branch cut for the logarithm. The gluing map for this fiber bundle at the branch cut increases the total sum of the liftings $s_i$ by $1$, which results in a coordinate transformation
$$
(s_1,\ldots,s_r)\mapsto \left(s_r-1+\frac1r,s_1+\frac1r,s_2+\frac1r,\ldots,s_{r-1}+\frac1r  \right),
$$
after normalizing so that $\sum s_i=0$. Therefore 
 $$
 (u_1,\ldots, u_{r}) \mapsto (u_{r}, u_1,\ldots, u_{r-1}),
 $$
since $s_1-s_r+1=1- u_1-\ldots-u_{r-1}$.
\end{proof}

The description of the last gluing map provides the following.

\begin{cor}
$\tilde{\mathcal{Y}}_r^{\textup{TR}}$ is orientable for odd $r$, non-orientable for even $r$. 
For $r=2$ it is a M\"obius band.
\end{cor}



From the proof of Proposition \ref{prop:totally-reducible-Ur}, that the fiber map $\Sym^r(S^1) \to S^1$ is exactly the determinant map. Hence, if we fix the determinant, we directly get the following description of the totally reducible $\SU(r)$-representations.

\begin{cor}
We have an isomorphism ${\mathcal{Y}}^{\textup{TR}}_r \cong \Delta_{r-1}$.
\end{cor}

Some results of this section are also proved in \cite[pag.\ 467]{florentino2009topology} by analysing the Lie algebra of ${\mathfrak{su}}(r)$.

\subsection{The irreducible locus}\label{sec:irreducible}

Consider an irreducible representation $(A, B) \in \tilde{\cY}_r^\ast$. We can conjugate $(A,B)$ into its standard form (\ref{eq:standard-form}) so the eigenvalues of $A$ are $\lambda_1, \ldots, \lambda_r$. Let $\nu_1, \ldots, \nu_r$ be the eigenvalues of $B$ so the map $\tilde{H}: \tilde{\cY}_r^\ast \to \Sym^r(S^1) \times \Sym^r(S^1)$ is $(A, B) \mapsto (\{\lambda_1, \ldots, \lambda_r\}, \{\nu_1, \ldots, \nu_r\})$.

\begin{prop}\label{prop:eigenvalues}
Let $(A,B) \in \tilde{\cY}_r^\ast$ with eigenvalues $\lambda_i$ and $\nu_i$, respectively. Then 
there exists $\varpi \in S^1$ such that $A^n = B^m = \varpi \Id$. Moreover, there exist unique $t_1, \ldots, t_r \in S^1$ such that $t_i^{nm}=\varpi$, $\lambda_i = t_i^m$ and 
$\nu_i = t_i^n$, for all $1 \leq i \leq r$. If, in addition, $(A,B) \in {\cY}_r^\ast$ then there exist finitely many choices for the eigenvalues $\lambda_i, \nu_i$.
\end{prop}

\begin{proof}
Let $(A,B) \in\tilde{\mathcal{Y}}^{\ast}_{r}$. Since $A^nB=B^mB=BB^m=BA^n$ we have that $A^n$ is an 
$\U(r)$-equivariant map. Hence, by Schur 
lemma we have $A^n = \varpi \Id$ for some $\varpi \in \CC^*$ and, a fortiori, since $A^n \in \U(r)$ also $\varpi \in S^1$. Obviously, we also have $B^m = A^n = \varpi \Id$.

As a consequence, if $\lambda_1, \ldots, \lambda_r$ and $\nu_1, \ldots, \nu_r$ are the eigenvalues of $A$ and $B$, we have
 \begin{align*}
  \lambda_j=\epsilon_j \lambda_1, \quad \epsilon_j\in \bm{\mu}_n, \quad j\geq 2, \\
  \nu_j=\varepsilon_j \nu_1, \quad \varepsilon_j\in \bm{\mu}_m, \quad j\geq 2.
  \end{align*}
Here $\bm{\mu}_s$ denotes the group of $s$-th roots of unity.  In particular, $\lambda_1=t^m$ and $\nu_1=t^n$, for some 
$t\in S^1$. As $\gcd(n,m)=1$, this $t$ is unique. Clearly $t^{nm}=\varpi$.

If the representation is in $\cY^{\ast}_r$, that is, an $\SU(r)$-representation, then $\varpi^r = \det(A^n) = 1$, so $\varpi \in \bm{\mu}_r$ . This implies that $t^{nmr} = \varpi^r = 1$, thus $t \in \bm{\mu}_{nmr}$. This gives finitely many choices for the eigenvalues of $A$ and $B$.
\end{proof}

\begin{rem}
The element $t$ only depends on a given choice of $\lambda_1,\nu_1$, that is, a given ordering of the eigenvalues.
\end{rem}

We can stratify $\Sym^r(S^1)$ according to the number of coincident eigenvalues. To be precise, given a partition $\tau = (r_1,\overset{(a_1)}{\ldots},r_1,\ldots,r_s,\overset{(a_s)}{\ldots},r_s) \in \Pi_r$, let us denote by $\Sym^r_{\tau}(S^1)$ the collection of sets $\{\lambda_1, \ldots, \lambda_r\}$ such that there are $a_1$ groups of $r_1$ equal eigenvalues, $a_2$ groups of $r_2$ equal eigenvalues and so on. In particular, $\tau_0 = (1,\overset{(r)}{\ldots},1)$ corresponds to different eigenvalues. In this way, given partitions $\tau_1, \tau_2 \in \Pi_r$, we set
$$
	\tilde{\cY}_{\tau_1, \tau_2}^\ast = \tilde{H}^{-1}(\Sym^r_{\tau_1}(S^1) \times \Sym^r_{\tau_2}(S^1)).
$$
These are isomorphism classes of representations $(A,B)$ such that $A$ has coincident eigenvalues given by $ \tau_1$ and $B$ has coincident eigenvalues given by $\tau_2$. Each representation belongs to one of these sets, and the associated pair of partitions $(\tau_1, \tau_2)$ is called the \emph{type} of the representation. Hence, we get a stratification
\begin{equation}\label{eq:strata-type}
	\tilde{\cY}_{r} = \bigsqcup_{\tau_1, \tau_2 \in \Pi_r} \tilde{\cY}_{\tau_1, \tau_2}^\ast.
\end{equation}
Notice that some of the strata $\tilde{\cY}_{\tau_1, \tau_2}^\ast$ may be empty in this decomposition. Moreover, this shows that the maximal component of $\tilde{\cY}_{r}^\ast$ corresponds to the type $(\tau_0, \tau_0)$.

\begin{thm}\label{thm:orthant}
The fiber of the map $\tilde{H}: \tilde{\cY}_{\tau_0, \tau_0} \to \Sym^r_{\tau_0}(S^1) \times \Sym^r_{\tau_0}(S^1)$
is isomorphic to 
  $$
  \bm{F}_r:= (S^1)^r \backslash \U(r) \,\slash (S^1)^r,
   $$ 
where the first $(S^1)^r$ acts by row multiplication and the second one by column multiplication. Furthermore, there is a surjective map
$\varphi:\bm{F}_r\to B$, where $B$ is the closed orthant $B^{r-1}\subset S^{r-1}$, defined as
$$
	B^{r-1} = \left\{(b_1, \ldots, b_r) \in \mathbb{R}_{\geq 0}^r \,\left| \, b_{1}^2+\ldots +b_{r}^2=1 \right.\right\}.
$$
satisfying the following properties:
\begin{itemize}
	\item Let $B_0 \subset B$ be the interior of the orthant. Then $\varphi^{-1}(B_0)$ is solely composed of irreducible representations and $\varphi|_{\varphi^{-1}(B_0)}$ is a trivial fiber bundle with fiber $\mathbb{C}P^{r-2}\rtimes\mathbb{C}P^{r-3}\rtimes \ldots \rtimes\mathbb{C}P^1$ 
	(that is, an iterated bundle of these spaces).
	\item The preimage of the boundary of the orthant contains both reducible and irreducible representations. 
\end{itemize}
\end{thm}

\begin{proof}
Consider $(A, B) \in \tilde{\cY}_{\tau_0, \tau_0}$ with fixed different eigenvalues and let us write it as in (\ref{eq:standard-form}) where $B$ is decomposed as in (\ref{eq:standard-form-eigen}). There is an action of $(S^1)^r$ on the eigenvectors of $A$, that rescales by  $S^1$ each of the rows of $P$ (a left action by multiplication on $\U(r)$), and an action of $S^1$ on the eigenvectors of $B$, that rescales each column (a right action on $\U(r)$). In this manner, we have that the fiber is contained in
\begin{equation} \label{eq:orthantorbit}
	(S^1)^r \backslash \U(r) \,\slash (S^1)^r.
\end{equation}
In this setup, the irreducible fiber $\tilde{\cY}^{\ast}_{\tau_0,\tau_0}$ is the collection of such classes of matrices where there exist no $s\times(r-s)$ zero minors in $B$, since otherwise the subspace generated by some of the eigenvectors of $A$ and $B$ would be invariant.

Let us describe the open set of this double quotient $(S^1)^r \backslash \U(r) \,\slash (S^1)^r$ where the first column vector has all its component non-zero. First, note that such representations will be irreducible by Proposition \ref{prop:semi-simple}. Because of the left $(S^1)^r$-action we may assume that the first column $(p_{11},p_{21},\ldots,p_{r1})\in \mathbb{R}^r_{>0}$, 
with $\sum p_{i1}^2=1$, defining an orthant $B^{r-1}_0$ of $S^{r-1}$, homeomophic to an open $(r-1)$-ball. The second eigenvector is orthogonal to $v_1$ and unitary, so it belongs to $S^{2r-3}\subset \mathbb{C}^{r-1}$. It is uniquely determined up to the $S^1$-action, 
hence it belongs to $S^{2r-3}/S^1\cong \mathbb{C}P^{r-2}$. 
The same construction shows that $w_3\in S^{2r-5}/S^1\cong \mathbb{C}P^{r-3}$, 
$w_4\in \mathbb{C}P^{r-4}$ up to $w_r$, which is uniquely determined up to the $S^1$-action 
on the last column. The real dimension of this set is $(r-1)+2(r-2)+\ldots+2=(r-1)^2$, which matches the dimension of (\ref{eq:orthantorbit}).
\end{proof}

Every reducible representation with eigenvalues of type $(\tau_0,\tau_0)$ will appear as a representation 
that belongs to the boundary of the orthant. For instance, in any case, the vertices $b_i = 1$ of $B$ are reducible representations, 
since they correspond to the case in which the first eigenvector of $B$, $p_1$, agrees with a canonical basis vector, that is, an eigenvector of $A$, hence providing a $1$-dimensional invariant subspace. Moreover, the fiber over a vertex is actually isomorphic to the space $\bm{F}_{r-1}=
(S^1)^{r-1}\backslash \U(r-1)/(S^1)^{r-1}$.
The set of totally reducible representations $\tilde{\mathcal{Y}}_r^{TR}$ appears at
the corners $b_i$ of the orthant $B$. 

In general, a reducible representation will occur when there is a $k$-dimensional invariant subspace $H=\langle e_{a_1},\ldots,e_{a_k} \rangle=\langle p_{b_1},\ldots,p_{b_k} \rangle$, which gives rise to a identically zero sub-minor in $P$ corresponding to the indices $(\lbrace 1,\ldots,r \rbrace- \lbrace a_1,\dots,a_k \rbrace)\times \lbrace b_1,\ldots,b_k \rbrace$. Because of Proposition \ref{prop:semi-simple}, the induced representation on $H^{\perp}$ will also determine 
an $(r-k)$-dimensional sub-representation. 
In the case that the representation is sum of two irreducible representations, that is corresponding to $\pi=(k,r-k)$, then the corresponding
subset of the fiber is given by 
 \begin{equation}\label{eq:orthred}
  \bm{F}_{k} \times \bm{F}_{r-k} \, .
 \end{equation}

\section{$\U(2)$-character variety}\label{sec:u2-char}

For $\U(2)$, we have two possible partitions: $(2)$ and $(1,1)$.

\begin{enumerate}
\item\label{enum:U2-irred} The partition $\pi=(2)$ corresponds to the irreducible representations. Observe that, in this case, both the eigenvalues of $A$ and $B$ must be different. Hence, by Theorem \ref{thm:orthant}, we know that the fiber of the map 
$\tilde{H}: \tilde{\cY}_2^\ast \to \Sym^2(S^1) \times \Sym^2(S^1)$ is contained in the orthant 
$B_0 = \{p_1^2 + p_2^2 = 1 \,|\, p_1, p_2 > 0\}$, which is isomorphic to the segment $(0,1)$. Furthermore, in the closure of this orthant, the only reducible representations correspond to the endpoints of the orthant and thus, the fiber of the eigenvalue map is exactly the open segment $(0,1)$.
	
By the proof of Proposition \ref{prop:eigenvalues}, the eigenvalues of $A$ are of the form $(\lambda,\lambda\epsilon)$ with
 $$
  \epsilon\in \bm{\mu}_n^*=\bm{\mu}_n-\{1\}.
  $$
As we can swap the eigenvalues, and this sends $\epsilon\to\epsilon^{-1}$,
we can assume 
 $$
 \epsilon\in \bm{\mu}_n^+=\{z\in \bm{\mu}_n^*|\Im z\geq 0\}. 
 $$
In the case that $n$ is even and $\epsilon=-1$,
we still have an action by swapping the eigenvectors of $A$.
Similarly, the eigenvalues of $B$ are $(\nu,\nu\varepsilon)$, with $\varepsilon\in \bm{\mu}_m^+$. Again by Proposition \ref{prop:eigenvalues}, there exists a unique $t \in S^1$ such that $t^m = \lambda$ and $t^n = \nu$,
so we get that the image of the eigenvalue map is $S^1 \times \bm{\mu}_n^+ \times \bm{\mu}_m^+$.

In the case that $n, m$ are both odd, we thus get that the eigenvalue map gives a trivial fibration $\tilde{H}: \tilde{\cY}_2^\ast = (0,1) \times S^1 \times \bm{\mu}_n^+ \times \bm{\mu}_m^+ \to S^1 \times \bm{\mu}_n^+ \times \bm{\mu}_m^+ $. Hence, we get 
 $\frac14 (n-1)(m-1)$ copies of the cylinder 
  $$
   \bm{C}:=(0,1)\x S^1.
   $$

Now suppose $n$ even and $m$ odd (the reverse case is similar by swapping $n,m$). Then there are $\lfloor\frac12(n-1)\rfloor \frac12(m-1)$
components $\bm{C}=(0,1) \x S^1$ as before. However, for $\epsilon=-1$, there is a residual action 
$(\lambda,-\lambda) \mapsto (-\lambda,\lambda)$
with $(p_{11},p_{21})\to (p_{21},p_{11})$. This changes $t\mapsto -t$, $\nu=t^n$ remains fixed, and $p_{11} \mapsto \sqrt{1-p_{11}^2}$. 
Hence, we get a quotient of $(0,1)\x S^1$ under $(p_{11},t) \sim (1-p_{11},-t)$, which is an open M\"obius band
 $$
  \bm{M} := ((0,1)\x S^1) / (p_{11},t) \sim (1-p_{11},-t) .  
  $$ 
There are $\frac12(m-1)$ components like these.

	\item The partition $\pi = (1,1)$ corresponds to totally reducible representations. In this case, we get by Proposition \ref{prop:totally-reducible-Ur} that $\tilde{\cY}_2^{\textup{TR}} = \Sym^2(S^1)$ is an $S^1$-bundle with fiber
	$$
		\Delta_1 = \left\{(u_1, u_2)\in \mathbb{R}^2 \, \left| \, u_2 = -u_1 + 1, \quad u_1, u_2 \geq 0  \right.\right\}.
	$$
	This is a closed segment parametrized by $u_1 \in [0,1]$. Analyzing the gluing map of Proposition \ref{prop:totally-reducible-Ur} 
	we get that $\tilde{\cY}_2^{\textup{TR}}$ is a closed M\"obius band. 
\end{enumerate}

In this way, the global picture of $\tilde{\cY}_2$ is the following. We have a M\"obius band of totally reducible representations $\bm{R}$. In the case of $n$ and $m$ odd, this band is intersected by $\frac14 (n-1)(m-1)$ cylinders corresponding to the irreducible representations glued through their boundaries $\{0,1\} \times S^1$. If $n$ is even, we get $\lfloor\frac12(n-1)\rfloor\, \frac12(m-1)$ cylinders, 
and also $\frac12(m-1)$ M\"obius bands, all glued to the base band $\bm{R}$ through their boundaries.



If we consider an irreducible representation in $\bm{C}$ for $A\sim(t^m, t^m\epsilon)$, $B\sim (t^n,t^n\varepsilon)$,
with $\epsilon\in \bm{\mu}^+_n$, $\varepsilon\in \bm{\mu}^+_m$, then one of the two reducible representations at the
boundary is given by $\lbrace t, t\alpha \rbrace$, $\alpha^m=\epsilon, \alpha^n=\varepsilon$, $\alpha\in \bm{\mu}_{mn}$. This goes around the longitude of $\bm{R}$, not the central one, and it goes around twice. This also applies to the other reducible representation at the boundary, given by $\lbrace t,t\alpha'\rbrace,$ where now $(\alpha')^m=\epsilon$, $(\alpha')^n=\varepsilon^{-1}$. Note that replacing $\alpha$ by $\alpha^{-1}$ or $\alpha'$ by $(\alpha')^{-1}$ yields the same  pair of reducible representations. 

As a consequence, the two circles of reducible representations at the boundary of $\bm{C}$ both inject into $\bm{R}$ and they do not intersect. Neither do circles that arise from different cilinders. When $n$ is even, $\epsilon =-1$ produces M\"obius band components of type $\bm{M}$. There is a single boundary circle $\lbrace t,t\alpha \rbrace$, where now $\alpha^m=-1, \alpha^n=\epsilon$, that also injects into $\bm{C}$.

\section{$\SU(3)$-character variety} \label{sec:SU3char}

\subsection{Geometric description of each stratum}\label{sec:geometric-description}

For $n=3$, we have three possible partitions: $(3)$, corresponding to the irreducible representations; $(2,1)$, corresponding to sums of
an irreducible $2$-dimensional representation an a $1$-dimensional representation; 
and $(1,1,1)$, corresponding to the totally reducible representations.

\begin{enumerate}
\item[\textbf{(1)}] For $\pi=(1,1,1)$, we deal with ${\mathcal{Y}}^{\textup{TR}}_3$, which is given by $\{t_1,t_2,t_3\} 
\in \Sym^3(S^1)$ such that $t_1t_2t_3=1$. By Proposition \ref{prop:totally-reducible-Ur}, the space is isomorphic to the simplex:
$$
\Delta_3 = \left\{(u_1, u_2, u_3)\in \mathbb{R}^3 \,\left|\, u_1 + u_2 + u_3 = 1, u_i\geq 0 \right.\right\}.
$$
which is a triangle $\bm{T}$ with vertices $(1,0,0),(0,1,0)$ and $(0,0,1)$. 



\item[\textbf{(2)}] For $\pi=(2,1)$, we have $\tilde Y_{\pi} =\tilde\cY_2^\ast \x \tilde\cY_1^\ast$. The determinant condition can be put in one
of the factors, for instance in the second one, thus $Y_{\pi}\cong \tilde\cY_2^\ast$. The description of this space was given in Section \ref{sec:u2-char}.

\item[\textbf{(3)}] For $\pi=(3)$, we deal with irreducible $\SU(3)$-representations $(A,B)$.
Let $\{\lambda_1,\lambda_2,\lambda_3\}$ be the eigenvalues of $A$, and
$\{\nu_1,\nu_2,\nu_3\}$ be the eigenvalues of $B$. We know that 
 $$
 \lambda_1^n=\lambda_2^n=\lambda_3^n=\nu_1^m=\nu_2^m=\nu_3^m=\varpi\in \bm{\mu}_3\, .
 $$
 
We stratify according to the type of the representation as in (\ref{eq:strata-type}).
If the three eigenvalues of $A$ (or of $B$) are equal, then the representation is automatically reducible.
If two eigenvalues of $A$ are equal, and also two eigenvalues of $B$ are equal, then intersecting the
two eigenplanes we get a common eigenvector for $A$ and $B$; and hence $(A,B)$ is reducible.
Therefore we have two cases:

\begin{enumerate}
\item[\textbf{(3a)}] All eigenvalues of $A$ are different, and all eigenvalues of $B$ are different.   
Fix one choice of eigenvalues, and focus on the corresponding space of irreducible representations.
By Theorem \ref{thm:orthant}, this is an open subset of $\bm{F}_3$. There is a map $\bm{F}_3 \to B^2$, the 
$2$-dimensional orthant defined as $B^2=\{(b_1,b_2,b_3)\in \mathbb{R}^3 \mid b_i\geq 0, b_1^2+b_2^2+b_3^2=1\}$, which is a triangle.
Over the interior, we have the fiber $\CC P^1$, defined by the second column of the matrix $P$ in (\ref{eq:standard-form-eigen}),
which is a projective point $[p_{21},p_{22},p_{23}]$, with $b_1p_{21}+b_2p_{22}+b_3p_{23}=0$. 

Over a vertex $b_i=1$, we only have reducible representations, in particular the totally reducible representations
plus some reducible representations of type $(2,1)$. Over a side of the orthant, say $(0,b_2,b_3)$,
we get some reducible representations of type $(2,1)$, and some irreducible ones. The space of irreducible
representations is parametrized by the second column of $P$, given by 
$(p_{21},p_{22},p_{23})$ with $b_2p_{22}+b_3p_{23}=0$. Therefore it is of the form
$(p,- q b_3, qb_2)$. It must be $q\neq 0$ and $p\neq 0$ by irreducibility (in the first case 
a column would have two zeroes, in the second case a row would have two zeroes).
Using the $S^1$-action on this vector, we can assume $q\in \RR_{> 0}$. Moreover, using the $S^1$-action on
the first row, we can take $p\in \RR_{>0}$. Finally $p^2+q^2=1$, so the fiber is parametrized by an open interval $(0,1)$.

%

\item[\textbf{(3b)}] Two eigenvalues of $A$ are equal. Then, the three eigenvalues of $B$ are different (it can be the other
way round, and this case should be counted as well). Now we look at a given component (fixing the eigenvalues).
We fix an orthonormal basis so that $B=\diag(\nu_1,\nu_2,\nu_3)$. The eigenspaces of $A$ are a vector $v=(a,b,c)$ and
the plane $P=\{ax+by+cz=0\}$ perpendicular to $v$. The coordinate vectors should not be $v$ or lie in $P$, hence
$a\neq 0$, $b\neq 0$ and $c\neq 0$. Moving with $S^1\x S^1\x S^1$ as before we can suppose $a,b,c >0$ with
$a^2+b^2+c^2=1$. This produces an open orthant $B^2_0$.

\end{enumerate}
 \end{enumerate}

\subsection{Closure of each strata} In this section, we shall study the closure of each of the strata analyzed in the previous section. 

\subsubsection*{Case (1). Representations of type $(1,1,1)$} This is a closed triangle $\bm{T}$
and thus no further representations appear in its closure.

\subsubsection*{Case (2). Representations of type $(2,1)$}

As described in Section \ref{sec:geometric-description} and \ref{sec:u2-char}, each of the components of $Y_{(2,1)}$ corresponds to a cylinder $\bm{C} = S^1 \times [0,1]$ or to a closed M\"obius band $\bm{M}$, depending on whether the eigenvalues $\lambda_{i_1} = - \lambda_{i_2}$ or not. In particular, no M\"obius bands appear if $n$ and $m$ are odd. In both cases, the boundary of the fibers correspond to totally reducible representations (Case (1)).

\subsubsection*{Case (3a). Irreducible representations with different eigenvalues}
Let us fix  $\bm{\lambda} = \{\lambda_1, \lambda_2, \lambda_3\}$ for $A$ and $\bm{\nu} = \{\nu_1, \nu_2, \nu_3\}$ for $B$ with all the eigenvalues pairwise different. 
In this case, by Theorem \ref{thm:orthant}, we have that the fiber of the eigenvalue map is
$$
	\bm{F}_3 = (S^1)^3 \backslash \U(3) / (S^1)^3 .
$$
Consider the map $\varphi: (S^1)^3 \backslash  \U(3) / (S^1)^3  \to B^2$ of Theorem \ref{thm:orthant}, corresponding to fixing the first eigenvector 
of $P$, where $B = B^2$ is a $2$-dimensional orthant. 
As described in Section \ref{sec:geometric-description}, on the interior of $B$, 
the component $F_{\bm{\lambda}, \bm{\nu}}$ is a trivial bundle with fiber $\CC P^1$. 
On the edges of the orthant, the intersection of the fiber with the irreducible representations is the interval $(0,1)$, as depicted in Figure \ref{fig:Irrrepmaximal}.

Let us study the closure of this set to understand the whole component $F_{\bm{\lambda}, \bm{\nu}} $. First of all, the three vertices $(1,0,0),(0,1,0),(0,0,1)$ of the triangle $B$ are respectively defined by the eigenvector condition $p_1=e_i$, $i=1,2,3$, so they correspond to reducible representations. On any of them, for example $(1,0,0)$, the quotient by the $S^1\times S^1$ actions on the second and third row allows us to assume $p_{22}^2+p_{32}^2=1$, where $p_{22},p_{32}\in \mathbb{R}_{>0}$, hence providing reducible representations of type $(2,1)$ isomorphic to an open interval $(0,1)$. Note that the closure of this interval, defined by the two limit cases $p_{22}=0$ and  $p_{32}=0$, provides two totally reducible representations in 
$\mathcal{Y}^{\textrm{TR}}_3$, 
which correspond to $p_2=e_2, p_3=e_3$ and $p_2=e_3, p_3=e_2$, as pictured in Figure \ref{fig:Irrrepmaximal}. The description is identical on any other given vertex.

\begin{figure}[h]
\includegraphics[width=7cm]{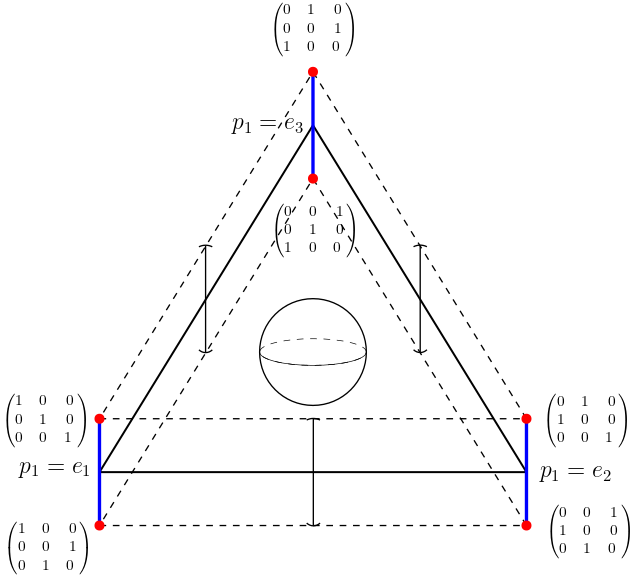}
\caption{A maximal dimensional component of irreducible representations.}
\label{fig:Irrrepmaximal}
\end{figure}

Secondly, on any side of  $B$, the set of irreducible representations for fixed $a,b>0$ was described in Section \ref{sec:geometric-description} as an open interval $(0,1)$ of equation $p^2+q^2=1$, where $p,q>0$. For instance, for the side $(0,b_2,b_3)$,  the second column is of the form $p_2=(p,-qb_3,qb_2)$,
and the other sides are similar.
The cases $p=0$ or $q=0$ lead again to reducible representations of type $(2,1)$, defined by the conditions  $p_3=e_i$ and $p_2=e_i$ respectively,
for the $i$-th side. Therefore, the closure over any side of the triangle 
$B$ are two open intervals $(0,1)$ of type $(2,1)$. 
Their closures provide four totally reducible representations, which were already accounted for over the vertices.

Furthermore, when we take into account all sides of $B$, the global compactification over the boundary is isomorphic to a closed Möbius band, as shown in Figure \ref{fig:closureglobal}. To check it, the condition $p_i=e_j$, 
$i,j=1,2,3$ that defines each of the nine partially reducible strata, 
follows the intersection pattern described in Figure \ref{fig:closureglobal}, connecting the 6 totally reducible representations over the vertices of $B$. Equivalently, we observe that the boundary of the fibers is a connected set, since we can interpolate from any permutation matrix to another through representations of type $(2,1)$. Since the only connected cover of degree $2$ over $S^1$ is the boundary of a M\"obius band, we get the result.

\begin{figure}[h]
\includegraphics[width=8cm]{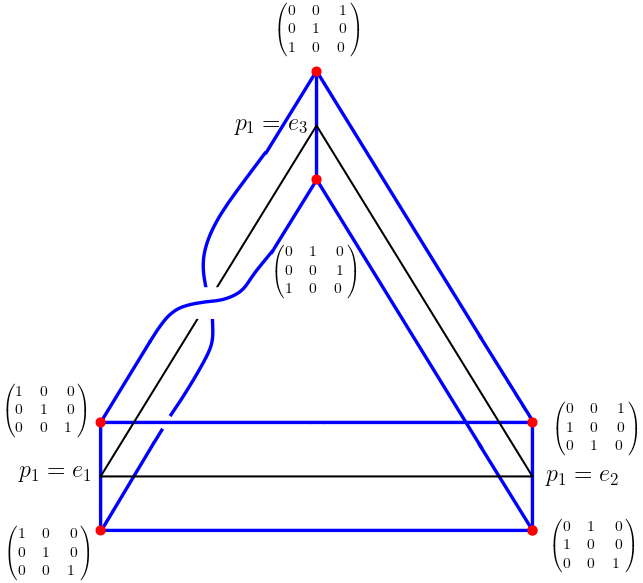}
\caption{Intersection pattern of $(2,1)$ and totally reducible representations.}
\label{fig:closureglobal}
\end{figure}

Notice that this picture agrees with the expected number of reducible representations to be found with eigenvalues $\bm{\lambda} = \{\lambda_1, \lambda_2, \lambda_3\}$ and $\bm{\nu} = \{\nu_1, \nu_2, \nu_3\}$ of $(A,B)$. 
There is an action of ${\mathfrak S}_3 \times {\mathfrak S}_3$ on these sets of eigenvalues by permutation. We may decompose this action as the composition of the ${\mathfrak S}_3$ free action on the eigenvalues of $A$ and the ${\mathfrak S}_3$ simultaneous action on both sets of eigenvalues, so that the latter coincides with the ${\mathfrak S}_3$ action on $\mathcal{Y}_{3}^{\textup{TR}}$. 
Therefore, for any fixed set of distinct eigenvalues $\{\lambda_1,\lambda_2,\lambda_3\}$ and $\{\nu_1,\nu_2,\nu_3\}$ we get:
\begin{itemize}
\item $6$ non-isomorphic totally reducible representations, which correspond to 6 points in $\cY_3^{\textrm{TR}}$, since the second ${\mathfrak S}_3$ action is trivial on 
$\mathcal{Y}_3^{TR}$.
\item $9$ components of reducible representations of type $(2,1)$, defined by the conditions $\lambda_i=\nu_j$, $i,j=1,2,3$ (or equivalently, by the eigenvector condition $p_i=e_j$, $i,j=1,2,3$). Any of these representations will set $p_{ij}=1$, forcing that $p_{il}=0$ for $l\neq j$ and $p_{lj}=0$ when $l\neq i$.
\end{itemize}

\subsubsection*{Case (3b). Irreducible representations with coincident eigenvalues}

Finally, any non-maximal dimensional irreducible stratum was characterized in Section \ref{sec:geometric-description} by two coincident eigenvalues for either $A$ or $B$. In that case, the space is isomorphic to an open disc and it is described by the equation $a^2+b^2+c^2=1$ where $v=(a,b,c)$ is the eigenvector perpendicular to the eigenplane $P$ and $a,b,c>0$. The closure is given by the three vertices $(1,0,0),(0,1,0),(0,0,1)$, which provide 3 totally reducible representations, and the three sides of the disc, which provide reducible representations of type $(2,1)$. Moreover, any point on the interior of a given side, say $(a,b,0)$, determines a single reducible representation, since in that case $(0,0,1)$ belongs to $P$ and defines a common eigenvector for $A$ and $B$. The compactification is isomorphic to a closed disc.
 
 \subsection{Component counting} \label{sec:counting}
 Let us count the number of components that each of the strata of the previous section.
 
 \begin{enumerate}
 	\item[\textbf{(1)}] $Y_{(1,1,1)} = \cY_3^{\textrm{TR}}$ is a connected space, homeomorphic to a triangle.
	\item[\textbf{(2)}] As described in Section \ref{sec:geometric-description}, $Y_{(2,1)}$ is the disjoint union of $\lfloor\frac12(n-1)\rfloor \frac12(m-1)$ cylinders $\bm{C}$ and, in the case when $n$ is even, $\frac12(m-1)$ M\"obius bands $\bm{M}$.
	\item[\textbf{(3a)}]  As in \cite{munozporti2016},
the action of ${\mathfrak S}_3$ on the $\{\lambda_1,\lambda_2,\lambda_3\}$ and
the action of ${\mathfrak S}_3$ on the $\{\nu_1,\nu_2,\nu_3\}$ are free, and the total number of possible eigenvalues is 
 $$
  3\left(\frac{1}{6}\right)^2(n-1)(n-2)(m-1)(m-2)= \frac{1}{12}(n-1)(n-2)(m-1)(m-2).
  $$
  	\item[\textbf{(3b)}] 
When the eigenvalues of $A$ are repeated, say of the form $\{\lambda,\lambda,\lambda\epsilon\}$, with $\epsilon\in \bm{\mu}_n^*$ and $\lambda^3=\epsilon^{-1}$, we have $\lambda\in\bm{\mu}_{3n}-\bm{\mu}_3$, which are $3n-3$ points. 
The eigenvalues of $B$ are of the form $\{\nu_1,\nu_2,\nu_3\}$ with $\nu_1^m=\nu_2^m=\nu_3^m=
\lambda^n$ and $\nu_1\nu_2\nu_3=1$. These are $(m-1)(m-2)$ points, and we have to quotient by the action of ${\mathfrak S}_3$.
The total number of components is thus $\frac12(n-1)(m-1)(m-2)$, as described in \cite{munozporti2016}. Symmetrically, when $B$ has repeated eigenvalues, we find another $\frac12(n-1)(n-2)(m-1)$ components (disjoint to the previous ones).
 \end{enumerate}
 
Recall that there are $\frac{1}{12}(n-1)(n-2)(m-1)(m-2)$ components of type (3a), and each of these components contains $6$ totally reducible representations in their closure. This gives rise to a total of $\frac{1}{2}(n-1)(n-2)(m-1)(m-2)$ intersection points with the set of totally irreducible representations. However, we can provide a different count of the number of totally reducible representations that appear in the closure of each of the strata (3a) and (3b) that will be useful for Section \ref{sec:intersection-tr}.
 
These intersection points were already characterized in Proposition \ref{prop:eigenvalues} for $r=3$ as those satisfying
 \begin{align*}
  \lambda_j=\epsilon_j \lambda_1, \quad \epsilon_j\in \bm{\mu}_n^*, \quad j=2,3, \quad \epsilon_2\neq \epsilon_3,\\
  \nu_j=\varepsilon_j \nu_1, \quad \varepsilon_j\in \bm{\mu}_m^*, \quad j=2,3, \quad \varepsilon_2\neq \varepsilon_3,
  \end{align*}
and $\lambda_1^n=\nu_1^m=\varpi\in \bm{\mu}_3$. Both triples of eigenvalues correspond to a single triple $\{t_1,t_2,t_3\}\in \Sym^3(S^1)$ such that
$$
t_1\in \bm{\mu}_{3nm}, \quad t_2=t_1e^{\frac{2\pi i k}{mn}} \quad t_3=t_1e^{\frac{2\pi i k'}{mn}}
$$
where $k,k'\not\equiv 0 $, $k\not\equiv k'$ (mod $ m$), $k,k'\not\equiv 0 $, $k\not\equiv k'$ (mod $n$), since all eigenvalues are different. Under the identifications made in Proposition \ref{prop:totally-reducible-Ur}, we get that these points are defined by pairs:
$$
u_1=\frac{k}{mn}, \quad u_2=\frac{k'-k}{mn}, \quad k'>k
$$
inside the simplex $\Delta_2$ of equation $u_1+u_2+u_3=1$, $u_i\geq 0$. If we forget about the condition $k'>k$, we are left with $(n-1)(m-1)$ choices for $k$ and also $(n-1)(m-1)$ choices for $k'$, from which we may substract from the latter $n-1$ and $m-2$ points from the conditions $k\equiv k'$ (mod $m$), $k\equiv k'$ (mod $n$) respectively. Thus, the total count equals $(n-1)(m-1)(n-2)(m-2)$ points. Since $k'>k$, we only need to consider half of them, giving a total of $\frac{1}{2}(n-1)(n-2)(m-1)(m-2)$, as expected.  

Furthermore, if we solely fix one of the three conditions $k\equiv 0$, $k\equiv 0$ or $k\equiv k'$ (mod $n$), we are left with $\frac{1}{2}(n-1)(m-1)(m-2)$ choices for such $k,k'$ pairs, giving a total of $\frac{3}{2}(n-1)(m-1)(m-2)$ points. They correspond to the totally reducible representations that appear as the $3$ vertices of the closure of each of the $\frac{1}{2}(n-1)(m-1)(m-2)$ components of type (3b) for coincident eigenvalues for $A$. Analogously, the case of coincident eigenvalues for $B$ is similar and provides $\frac{3}{2}(m-1)(n-1)(n-2)$ points, and thus in total we get 
$\frac32 (n-1)(m-1)(n+m-4)$ 
intersection points.


\section{Intersection with totally reducible representations}\label{sec:intersection-tr}

In this section, we shall consider the eigenvalue map
$$
	H: \cY_3 \to \Sym^3(S^1) \times \Sym^3(S^1).
$$
The fiber of this map depends on the arithmetic of the eigenvalues $\bm{\lambda} = \{\lambda_1, \lambda_2, \lambda_3\}$ of $A$ and $\bm{\nu} = \{\nu_1, \nu_2, \nu_3\}$ of $B$. Set $F_{\bm{\lambda}, \bm{\nu}} = H^{-1}(\bm{\lambda}, \bm{\nu})$. We may distinguish three (not necessarily disjoint) cases:
\begin{enumerate}
	\item[\textbf{(1)}] $F_{\bm{\lambda}, \bm{\nu}}$ only contains totally reducible representations i.e. of type $\pi = (1,1,1)$.
	\item[\textbf{(2)}] $F_{\bm{\lambda}, \bm{\nu}}$ contains representations of type $\pi = (2,1)$. In that case, there must exist different indices $1 \leq i_1, i_2 \leq 3$ and $1 \leq j_1, j_2 \leq 3$ such that $\lambda_{i_1}^n = \lambda_{i_2}^n = \nu_{j_1}^m = \nu_{j_2}^m$. Hence, we have $\lambda_{i_2} = \lambda_{i_1} \epsilon$ and $\nu_{j_2} = \nu_{j_1} \varepsilon$ for $\epsilon \in \bm{\mu}_n^*$ and $\varepsilon \in \bm{\mu}_m^*$, 
	as well as $\lambda_{i_1} = t^m$ and $\nu_{j_1} = t^n$ for some $t \in S^1$. Since the value of the third eigenvalue is determined, this provides finitely many copies of circles in $\Sym^3(S^1) \times \Sym^3(S^1)$, maybe containing points of case (1).
	\item[\textbf{(3)}] $F_{\bm{\lambda}, \bm{\nu}}$ contains irreducible representations. In that case, there must exists $\varpi \in \bm{\mu}_3$ 
	such that $\lambda_i^n = \nu_j^m = \varpi$ for all $i,j$. These are finitely many points in the base which can be also subdivided into:
	\begin{enumerate}
		\item[\textbf{(3a)}] All the eigenvalues $\lambda_i$ are pairwise different, and also all the eigenvalues $\nu_j$. This corresponds to type $\pi = (3)$, 
		case (3a), of Section \ref{sec:geometric-description}.
		\item[\textbf{(3b)}] Two of the eigenvalues $\nu_j$ are equal (or two of the eigenvalues $\lambda_i$ are equal, but not the two conditions
		simultaneously). This corresponds to type $\pi = (3)$, case (3b), 
		of Section \ref{sec:geometric-description}.
	\end{enumerate}
\end{enumerate}

Recall from Section \ref{sec:geometric-description} that the set of totally reducible representations is a $2$-dimensional triangle $\bm{T}$. 
The fiber under the eigenvalue map
$$
	H: \bm{T} \to \Sym^3(S^1) \times \Sym^3(S^1)
$$
is up to $6$ points, depending on how many eigenvalues are repeated. This triangle $\bm{T}$ intersects with the other representations as follows:
\begin{itemize}
	\item $\frac{1}{2}(n-1)(n-2)(m-1)(m-2)$ points of $\bm{T}$ are mapped into strata (3a) under $H$. 
	Each of the components of case (3a) intersects exactly $6$ of these points of $\bm{T}$ through their vertices.
	\item $\frac{3}{2}(n-1)(m-1)(m-2)$ points of $\bm{T}$ are mapped to strata (3b) under $H$, 
	corresponding to repeated eigenvalues in $A$, and another $\frac{3}{2}(m-1)(n-1)(n-2)$ points correspond to repeated eigenvalues in $B$. The components of case (3b), 
	which are triangles, intersect $\bm{T}$ in $3$ points (the $3$ totally reducible representations they contain).
	
	\item There are $\frac12 (n-1)(m-1)$ circles in $\bm{T}$ which are mapped to strata (2). In that case, the components of 
	case (2) intersect these circles through their boundary (two different circles in the case of cylinders $\bm{C}$ 
	and one circle in the case of M\"obius bands $\bm{M}$).
	
	\item The circles corresponding to strata (2) may contain points of strata (3a) and (3b). 
	To be precise, let us consider the case when $m,n$ are odd. The eigenvalues of each of the circles of strata (2) are of the form 
$$\{t^m, t^m\epsilon, t^{-2m}\epsilon^{-1}\}, \quad \{t^n, t^n\varepsilon, t^{-2n}\varepsilon^{-1}\},$$
for $A$ and $B$ respectively, where $\epsilon \in \bm{\mu}_n^+$ and $\varepsilon \in \bm{\mu}_m^+$ are fixed ($t \in S^1$ is varying). Hence, we get that $(t^{-2m}\epsilon^{-1})^n = t^{nm}$ if and only if $t \in \bm{\mu}_{3nm}$. 
Therefore, the circle contains points of strata (3a) and (3b). 
Removing the cases in which the third eigenvalue coincides with the first or the second one, 
this gives $3mn-6m-6n+12$ points of case (3a) on each of this circles. There are also $6m+6n-24$ points of case (3b).

These circles can be explicitly described in $\bm{T}$ as follows: 
taking $\alpha=e^{\frac{2\pi i k}{mn}}\in \bm{\mu}_{mn}$ so that $\alpha^m=\epsilon, \alpha^n=\varepsilon$, each circle is given by
$$
\{ t,t\alpha, t^{-2}\alpha^{-1} \}\in \Sym^3(S^1),
$$
where $t\in S^1$. As $t$ varies, we can consider that $t$ and $t\alpha$ travel once counterclockwise around $S^1$  whereas $t^{2}\alpha^{-1}$ winds $S^1$ twice in the opposite direction. Taking into account the identifications made in Proposition \ref{prop:totally-reducible-Ur} for $t$ and $t\alpha$, either $u_1,u_2$ or $u_3$ will be constant and equal to $\frac{k}{mn}$ or $1-\frac{k}{mn}$. More concretely, the circle in $\bm{T}$ 
is a piecewise linear path that meets the sides of $\bm{T}$ 
at $6$ points, which are precisely those where 
$t^{-2}\alpha^{-1}$ equals either $t$ or $t\alpha$, which happens three times as $t,t\alpha$ and $t^{-2}\alpha^{-1}$ move around $S^1$. At these points the lifts of $t,t\alpha$ and $t^{-2}\alpha^{-1}$ are permuted to satisfy the normalization condition described in Equation (\ref{eq:redsimplex}).  Starting at $u_1=u_2=\frac{k}{mn}$ when $t=1$, the curve is described in Figure \ref{fig:circlesinT}.  

\begin{figure}[h]
\includegraphics[width=11cm]{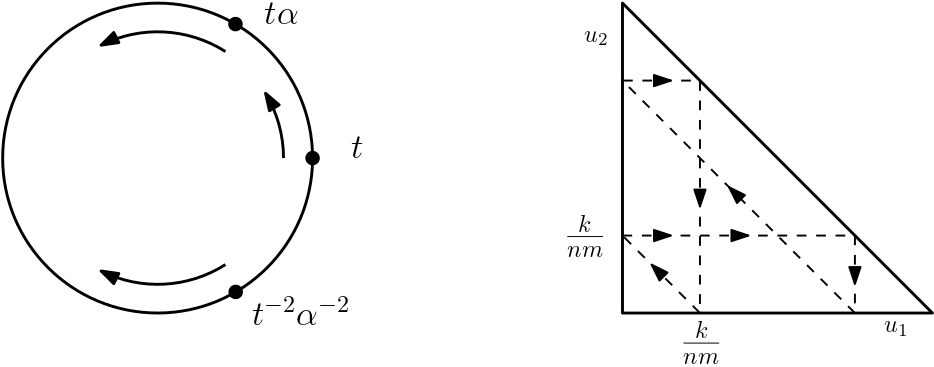}
\caption{Circles in $\bm{T}$ that are mapped to strata (2).}
\label{fig:circlesinT}
\end{figure}

There are $3mn$ rational points of the form $(\frac{k}{mn},\frac{k'}{mn})$ on each of these circles, and $6$ of them  always correspond to reducible representations lying on the sides of $\bm{T}$. 
Inside, we obtain points of cases (3a) and (3b), the latter taking place when only one of the $6$ equations:
$$
k\equiv 0, \qquad k'\equiv 0, \qquad k+k'\equiv 0 \qquad (\text{mod }m),\: (\text{mod }n)
$$
is satisfied. Reducible representations appear when more than one equation is true for the pair $(k,k')$.

\begin{figure}%
    \centering
    \subfloat[\centering Gluing points of cases (1), (2), (3a) and (3b) in $\bm{T}$.]{{\includegraphics[width=6cm]{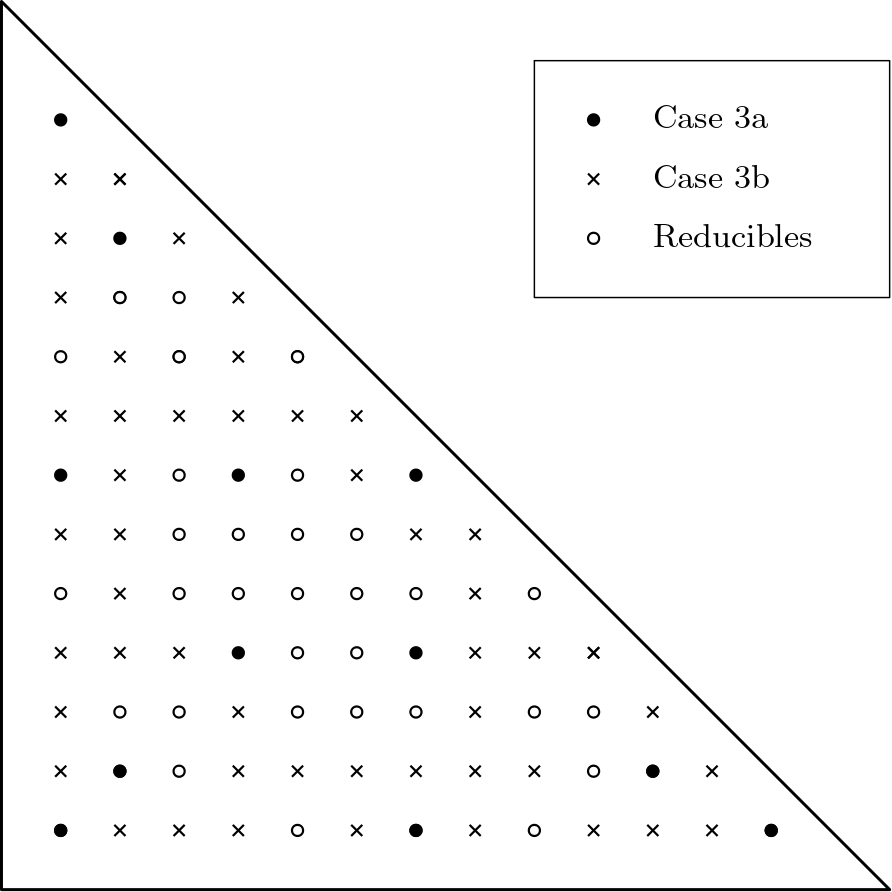} }}%
    \qquad
    \subfloat[\centering A circle in $\bm{T}$, $\alpha=\exp(2 \pi i k/mn)$.]{{\includegraphics[width=6cm]{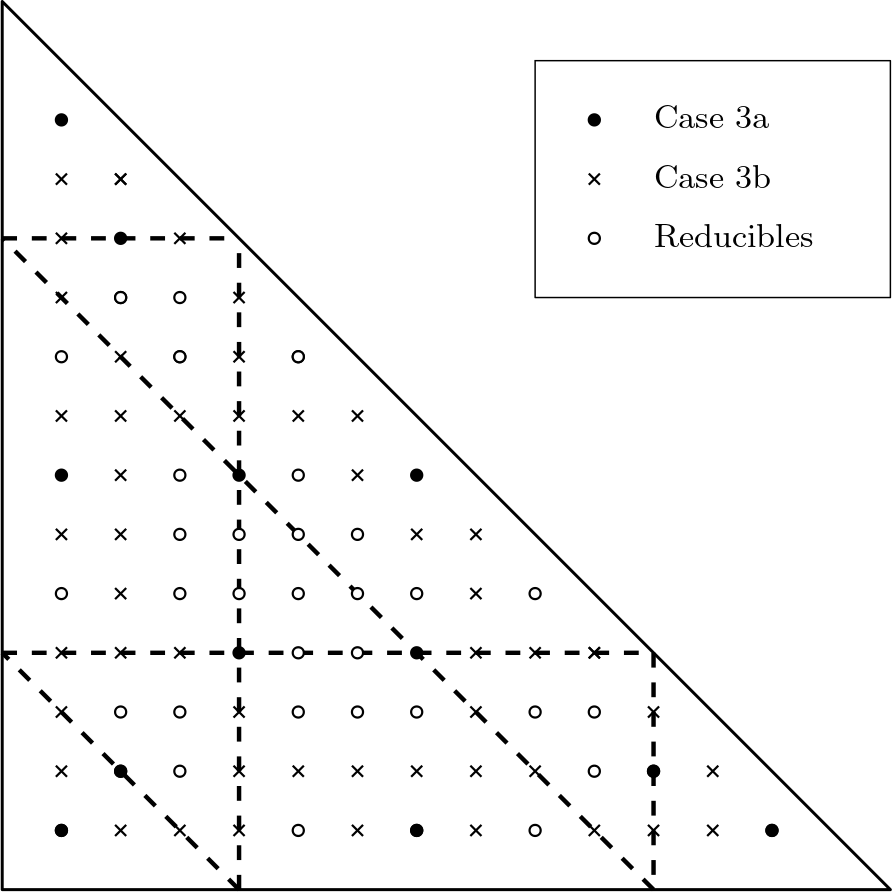} }}%
    \caption{Circles in $T$ for $n=3,m=5$.}%
    \label{fig:gluingpointsinT}%
\end{figure}

When $n$ is even, $\frac{1}{2}(m-1)$ of these circles from strata (2) correspond to boundaries of M\"obius bands that appear in the even case. 
They are also given by 
$\lbrace t,t\alpha,t^{-2}\alpha^{-1} \rbrace \in \Sym^3(S^1)$, where now $\alpha^m=-1$, $\alpha^n=\epsilon \in \bm{\mu}_m^{\ast}$.

	\item From the previous description, observe also that the circles of strata (2) are not disjoint: they may be connected to another through one of these points of cases (3a) and (3b). 
	In fact, any point corresponding to case (3a) or (3b) is shared by exactly $3$ circles 
	(the $3$ possible choices of getting a $(2,1)$-representation from them).	
	\item Each time that one of these circles of strata (2) contains a pair of points of type (3b), the corresponding 
	fiber $[0,1]$ of the cylinder/M\"obius band is glued with the edge of a triangle of type (3b). 
	Observe that then each triangle of type (3b) is attached to $3$ cylinders/M\"obius bands corresponding to different circles.

	\item At each of the circles of strata (2) that contain a pair of points of type (3a), the corresponding fiber $[0,1]$ of the cylinder/M\"obius band is glued with a segment between two totally reducible representations of the boundary over the edges of a component of type (3a) (the blue segments in Figure \ref{fig:closureglobal}). Each of these components of type (3a) contains $9$ such segments, so the component is glued with $9$ cylinders/M\"obius bands. 
\end{itemize}

\section{Homological invariants}\label{sec:homological-invariants}

From the descriptions of Sections \ref{sec:SU3char} and \ref{sec:intersection-tr}, it is possible to compute several homological invariants of the $\SU(3)$-character variety.

\subsection{Euler characteristic with compact support}\label{sec:homological-invariants-euler}

In this section, we shall compute the Euler characteristic with compact support of $\cY_3$ in the case that both $n$ and $m$ are odd. Recall that given a topological space $X$, its Poincar\'e polynomial with compact support is
$$
	P_c(X) = \sum_{k} \dim_\RR H_c^k(X; \RR)\,t^k \in \ZZ[t],
$$
from which we can form the Euler characteristic with compact support as the integer
$$
	\chi_c(X) = P_c(X)(-1) = \sum_{k} (-1)^k \dim_\RR H_c^k(X; \RR) \in \ZZ.
$$
For our purpose, we shall exploit that this Euler characteristic is additive, that is,
$$
	\chi_c(X) = \chi_c(U) + \chi_c(X - U)
$$
for any open subset $U \subset X$. 
It is noticeable that it is also multiplicative for products $\chi_c(X \times Y) = \chi_c(X) \chi_c(Y)$.

Now, let us compute the Euler characteristic for each stratum of $\cY_3$.

\begin{enumerate}
	\item[\bf{(1)}] For representations of type $\pi = (1,1,1)$ we have $Y_{(1,1,1)} = \cY^{\textrm{TR}}_3$, which has $P_c(\cY^{\textrm{TR}}_3) = 1$ so $\chi_c(\cY^{\textrm{TR}}_3) = 1$.
	\item[\bf{(2)}] For representations of type $\pi = (2,1)$, we have that they are the open cylinders $\bm{C}$ and M\"obius bands $\bm{M}$. Both have compactly-supported Poincar\'e polynomials $P_c(\bm{C}) = P_c(\bm{M}) = t^2+t$ so $\chi_c(\bm{C}) = \chi_c(\bm{M}) = 0$. Hence, this stratum provides no contribution.
	\item[\bf{(3)}] For representations of type $\pi = (3)$, we have two possibilities:
		\begin{itemize}
			\item[\bf{(3a)}] In the case (3a), we have that each component of $Y_{(3)}$ is $\bm{F}_3 - \cY^{\textrm{TR}}_3 - Y_{(2,1)}$. 
			Let us decompose the orthant triangle $B$ into its interior $B_0$ and its boundary $\partial B = B - B_0$. 
			The fibration $\varphi: \bm{F}_3|_{\varphi^{-1}(B_0)} \to B_0$ 
			is trivial with fiber $S^2$, so $\chi_c(\varphi^{-1}(B_0)) = \chi_c(B_0)\chi_c(S^2) = 2$. 
Regarding the stratum $\varphi^{-1}(\partial B)$ , this is a closed M\"obius band so it has vanishing Euler characteristic with compact support. 
We have to remove the reducible representations of this set, which is the boundary of the M\"obius band and three segments, with Euler characteristics 
$\chi_c(\partial \overline{\bm{M}}) = \chi_c(S^1) = 0$ and $\chi_c((0,1)) = -1$. Hence $\chi_c(\varphi^{-1}(\partial B_0) \cap Y_{(3)}) 
= 0-0-3\cdot (-1) = 3$ and in 
total $\chi_c(\bm{F}_3 - \cY^{\textrm{TR}}_3 - Y_{(2,1)}) =5$ and  
$\chi_c(Y_{(3)}^{a}) = \frac{5}{12}(n-1)(n-2)(m-1)(m-2)$.
 \item[\bf{(3b)}] In the case (3b), we have that, for each component, the irreducible representations form an open $2$-dimensional triangle, with Euler characteristic $1$. Taking into account that there are $\frac12(n-1)(m-1)(n+ m - 4)$ of these triangles, 
 this gives a total contribution of $\chi_c(Y_{(3)}^{b}) = \frac12(n-1)(m-1)(n+ m - 4)$.
\end{itemize}
\end{enumerate}

Adding up all these contributions, we finally get 
$$
	\chi_c(\cY_3) = 1 + (n-1)(m-1) \left( \frac{n+ m - 4}{2} + \frac{5(n-2)(m-2)}{12} \right).	
$$

\begin{rem}
This Euler characteristic coincides with the one of $\cX_3 = X(\G_{n,m}, \SL(3, \CC))$, as obtained by setting $\mathbb{L} = 1$ in the formula of its motive \cite[Theorem 8.3]{munozporti2016}.
\end{rem}

\subsection{Homology for $n=2$ and $m$ odd}\label{sec:homological-invariants-homology}

In this section, we shall compute the homology of $\cY_3$ for $n = 2$ and $m > 2$ odd. In that case, from the counting of Section \ref{sec:counting}, we find that some strata are simplified:

\begin{itemize}
	\item In case (1), we get a triangle $\cY_3^{\mathrm{TR}}$, which is a contractible space.
	\item In case (2), there are no cylinders, and we find $N_1 = \frac{1}{2}(m-1)$ closed M\"obius bands $\overline{\bm{M}}$ attached to $\cY_3^{\mathrm{TR}}$ through copies of $S^1$, one per component. When collapsing 
	$\cY_3^{\mathrm{TR}}$, the M\'obius bands get their boundaries collapsed.
	\item There are no strata of case (3a) (irreducible with different eigenvalues).
	\item There are $N_2 = \frac12(m-1)(m-2)$ components of case (3b), which are triangles with edges attached to meridians of corresponding components $\overline{\bm{M}}$.
\end{itemize}

Let us take $X = \cY_3^{\mathrm{TR}} \cup Y_{(2,1)}$, 
that is, a disc with $N_1$ M\"obius bands attached through their boundaries. Collapsing the boundary of a M\"obius band is equivalent to attaching a disc to it, which converts it into a real projective plane. Hence, contracting $\cY_3^{\mathrm{TR}}$ we get that $X$ is homotopically equivalent to a bouquet of $N_1$ real projective planes, so its homology is $H_0(X) = \ZZ$, $H_1(X) = \ZZ_2^{N_1}$ and $H_2(X) = 0$.

Now, let $Y$ be the union of the closures of the $N_2$ triangles of case (3b). We have that $Y$ is a disjoint union of $N_2$ closed discs, so $H_0(Y) = \ZZ^{N_2}$, $H_1(Y) = H_2(Y) = 0$. The intersection $X \cap Y$ is the collection of edges of the triangles (3b) so they are a disjoint union of $N_2$ circles. Hence $H_0(X \cap Y) = H_1(X \cap Y) = \ZZ^{N_2}$ and $H_2(X \cap Y) = 0$.

Therefore, using the Mayer-Vietoris long exact sequence and observing that $X \cup Y$ is homotopically equivalent to $\cY_3$, we get
\begin{center}
\begin{tikzpicture}[descr/.style={fill=white,inner sep=1.5pt}]
        \matrix (m) [
            matrix of math nodes,
            row sep=1em,
            column sep=2.5em,
            text height=1.5ex, text depth=0.25ex
        ]
        { 0 & H_2(X \cap Y) = 0 & H_2(X) \oplus H_2(Y) = 0 & H_2(\cY_3) & \\
            & H_1(X \cap Y) = \ZZ^{N_2} & H_1(X) \oplus H_1(Y) = \ZZ_2^{N_1} & H_1(\cY_3) & \\
            & H_0(X \cap Y) = \ZZ^{N_2} & H_0(X) \oplus H_0(Y) = \ZZ^{N_2+1} & H_0(\cY_3) = \ZZ & 0\\
        };

        \path[overlay,->, font=\scriptsize,>=latex]
        (m-1-1) edge (m-1-2)
        (m-1-2) edge (m-1-3)
        (m-1-3) edge (m-1-4)
        (m-1-4) edge[out=355,in=175]  (m-2-2)
        (m-2-2) edge node[above]{$f$} (m-2-3)
        (m-2-3) edge (m-2-4)
        (m-2-4) edge[out=355,in=175]  (m-3-2)
        (m-3-2) edge (m-3-3)
        (m-3-3) edge (m-3-4)
        (m-3-4) edge (m-3-5);
\end{tikzpicture}
\end{center}

Notice that the map $f$ sends each boundary of a triangle to the three meridians of the M\"obius bands it intersects. A key fact is the following.

\begin{lem}\label{lem:f-surjective}
The map $f$ is surjective.
\end{lem}

\begin{proof}
There are $N_2= (m-1)(m-2)/2$ triangles.
Each triangle corresponds to representations of the form
 $(A,B)$, with $A\sim \diag(\lambda,\lambda, -\lambda)$ and $B\sim
\diag (\nu,\nu \epsilon,\nu\epsilon')$ where $\nu \in \bm{\mu}_{3m}$
and $\lambda^2=\nu^m$, and $\epsilon,\epsilon'\in \bm{\mu}_m$ with
$\epsilon\neq 1$, $\epsilon'\neq 1$, $\epsilon\neq \epsilon'$,
satisfying that $\nu^3=(\epsilon\epsilon')^{-1}$. Call
$\bm{T}_{\nu,\epsilon,\epsilon'}$ one of these triangles.

There are $N_1=(m-1)/2$ collapsed M\"obius bands, corresponding to
representations with $A_1\sim \diag(\lambda,-\lambda)$, $B_1\sim \diag(\nu,\epsilon\nu)$,
where $\epsilon \in \bm{\mu}_m^*$. And each of them is parametrized by
$\epsilon^{\pm 1}$, namely the ratio of the two eigenvalues of $B_1$.
Call $\ell_k$ the meridian of the M\"obius band 
corresponding to $\epsilon_k= e^{2\pi i k/m}$, $k\not\equiv 0 \pmod m$.
Note that $\ell_{-k}=\ell_k$.

When looking at the boundary of $\bm{T}_{\nu,\epsilon,\epsilon'}$,
we get three segments, corresponding to choosing two of the three eigenvalues of
$B$ and taking their ratio. The map $f$ is nothing but seeing this boundary as part of the homology of $H_1(X)$, so it is given by
$$
 f \left(\partial\bm{T}_{\nu,\epsilon_a,\epsilon_b}\right) = [\ell_a]+[\ell_b]+[\ell_{a-b}],
$$
where the coefficients on the right are understood modulo $2$. Here $a,b,a-b\not\equiv 0\pmod m$. Observe that
$$
 f \left(\partial\bm{T}_{\nu,\epsilon_{2a},\epsilon_{a}}\right) = [\ell_{2a}]+[\ell_a]+[\ell_{a}] = [\ell_{2a}].
$$
Since $m$ is odd, any element of $\ZZ_m$ can be written as $2a$ for some $a \in \ZZ_m$, showing that
the map $f$ is surjective.
\end{proof}

By the previous lemma, since $f$ is surjective we get $H_1(\cY_3) = 0$ and $H_2(\cY_3) = \textrm{Ker}\, f = \ZZ^{N_2}$. In this way, we get that the homology of $\cY_3$ is
\begin{equation}\label{eqn:H(Y)}
 H_0(\cY_3) = \ZZ, \quad H_1(\cY_3) = 0, \quad H_2(\cY_3) = \ZZ^{\frac{1}{2}(m-1)(m-2)}, 
 \quad H_k(\cY_3) = 0 \textrm{ for $k \geq 3$}.
\end{equation}

Furthermore, the previous computations gives us something a bit stronger. 

\begin{cor}\label{cor:Y3-simply-connected}
The character variety $\cY_3$ is simply-connected.
\end{cor}

\begin{proof}
Take $U, V \subset \cY_3$ open sets slightly bigger than the closed sets $X, Y \subset \cY_3$ respectively, such that $U$ retracts to $X$, $V$ retracts to $Y$, and $U \cap Y$ retracts to $X \cap Y$. We collapse the set of totally reducible representations $\cY_3^{\textrm{TR}} \subset \cY_3$ and we use it as basepoint for the fundamental groups. By the proof above, $U$ is equivalent to a bouquet of $N_1$ real projective planes, so $\pi_1(U) = \ZZ_2^{N_1}$. In the same vein, $V$ is simply-connected. The intersection has $\pi_1(U \cap V) = \ZZ^{N_2}$ and the map induced by the inclusion is precisely $f: \pi_1(U \cap V) = \ZZ^{N_2} \to \pi_1(U) = \ZZ_2^{N_1}$. By Lemma \ref{lem:f-surjective}, the morphsm $f$ is surjective so the fibered product $\pi_1(U) \star_{\pi_1(U \cap V) }  \pi_1(V)$ vanishes, which agrees with $\pi_1(\cY_3)$ by Seifert-van Kampen theorem.
\end{proof}

\subsection{The inclusion $\cY_3 \hookrightarrow \cX_3$}
Now we compare the $\SU(3)$-character variety $\cY_3 = X(\G_{n,m}, \SU(3))$ with the $\SL(3,\CC)$-character variety $\cX_3 = X(\G_{n,m}, \SL(3, \CC))$. From \cite{munozporti2016}, we observe that $\cX_3$ admits an analogous stratification to the one of $\cY_3$ presented in Section \ref{sec:SU3char}. With the same notation as in Section \ref{sec:SU3char}, for $\cX_3$ these strata are the following:
\begin{enumerate}
	\item[\textbf{(1)}] This corresponds to totally reducible representations. By \cite[Proposition 8.1]{munozporti2016}, this stratum in $\cX_3$ is isomorphic to $\CC$. In particular, it is contractible and thus homotopically equivalent to $\cY^{\textrm{TR}}_3$.
	\item[\textbf{(2)}] This corresponds to partially reducible representations. There are two cases:
		\begin{itemize}
			\item The analogues of the cylinders $\bm{C}$. By \cite[Proposition 8.1]{munozporti2016}, in $\cX_3$ they correspond to $(\CC - \{0,1\}) \times \CC^*$. This space is not homotopically equivalent to $\bm{C}$ since $H_2((\CC - \{0,1\}) \times \CC^*) = H_1(\CC - \{0,1\}) \otimes H_1(\CC^*) \neq 0$ but $H_2(\bm{C}) = 0$.
			\item The analogues of the M\"obius bands $\bm{M}$. By \cite[Proposition 8.1]{munozporti2016}, in $\cX_3$ they correspond to $\CC^2 - \{y = 0\} - \{y = x^2\}$. Using the Mayer-Vietoris exact sequence we find that this later space has $H_2(\CC^2 - \{y = 0\} - \{y = x^2\}) \neq 0$ but $H_2(\bm{M}) = 0$, showing that they are not homotopically equivalent.
		\end{itemize} 
	\item[\textbf{(3a)}] This is the space of irreducible representations with different eigenvalues. By \cite[Proposition 8.3]{munozporti2016}, in $\cX_3$ this corresponds to the quotient $\mathcal{M} / (T \times_D T)$. Here, $\mathcal{M} \subset \GL(3, \CC)$ is the stable locus of the action of $T \times_D T$ on $\GL(3, \CC)$, where $T = (\CC^*)^3$ is the set of diagonal matrices, the first factor acts on $\GL(3, \CC)$ by multiplication on the left, and the second by multiplication on the right. The subset $D = \{(\lambda \Id, \lambda^{-1} \Id)\,|\, \lambda \in \CC^*\} \subset T \times T$ is the subgroup of multiples of the identity acting trivially. Notices that, by setting $\mathbb{L}=1$ in its motive \cite[Proposition 8.3]{munozporti2016}, we get that its Euler characteristic with compact support agrees with the one of the stratum (3a) of $\cY_3$.
	\item[\textbf{(3b)}] This corresponds to irreducible representations with a coincident eigenvalue. By \cite[Proposition 8.2]{munozporti2016}, in $\cX_3$ they correspond to a $(\CC^*)^2 - \{x+y = 1\}$.
	Using the Mayer-Vietoris exact sequence, $H_2((\CC^*)^2 - \{x+y = 1\}) \neq 0$. But the corresponding strata in $\cY_3$ are contractible, so they are not homotopically equivalent. 
\end{enumerate}

This analysis evidences that, for $n,m > 2$, not all the strata are homotopically equivalent. However, it may happen that the inclusion $\cY_3 \hookrightarrow \cX_3$ is still a homotopy equivalence, even though such homotopy does not preserve the stratification by type of the representation. We show that this is indeed the case when $n=2$.

\begin{thm}
For $n=2$ and odd $m$, the inclusion $\cY_3 \hookrightarrow \cX_3$ is a homology equivalence.
\end{thm}

\begin{proof}
Let us compute the homology of $\cX_3$ for $n = 2$ and $m$ odd, following the strategy of Section \ref{sec:homological-invariants-homology}. As there, we take $\hat{X} = \cX_3^{\mathrm{TR}} \cup X_{(2,1)}$, which in this case is $\CC^2$ 
with $N_1 = \frac12 (m-1)$ copies of $\CC^2 - \{y = 0\} - \{y = x^2\}$ attached through the one of the missing lines. 
The closure of a component is $\CC^2-\{y=0\}=\CC\x \CC^*$, and we
glue to $\cX_3^{\mathrm{TR}}$ along $D=\CC^*=\{y=x^2\}$. When we collapse $\cX_3^{\mathrm{TR}}$,
we get $(\CC\x\CC^*)/D$ whose homology is given by 
 $H_k((\CC\x\CC^*)/D)= H_k(\CC\x\CC^*,D)$, for $k\geq 1$.   
  Using the exact sequence of the pair, and that the inclusion $H_1(D)=\ZZ \to H_1(\CC\x \CC^*)=\ZZ$ is multiplication by
  $2$ (this follows from the fact that the loop $x = e^{2 \pi i t}$ goes to $(x,y) = (e^{2 \pi i t}, e^{4 \pi i t})$
  that winds twice around $y=0$), we get 
 $$
  H_1((\CC\x\CC^*)/D)= \ZZ_2,\qquad  H_k((\CC\x\CC^*)/D)= 0, k\geq 2
 $$  
Therefore $H_1(\hat{X})=\ZZ_2^{N_1}$ and $H_k(\hat{X})=0$ for $k\geq 2$.

Now, let $\hat{Y}$ be the union of the closures of the $N_2 = \frac12 (m-1)(m-2)$ strata of case (3b). 
These are $N_2$ copies of $\CC^2$, and each copy is attached to $\hat{X}$ through three lines 
intersecting pairwise at three different points. 
Therefore 
$H_0(\hat{Y}) = \ZZ^{N_2}$ and $H_k(\hat{Y}) = 0$ for $k \geq 1$. The intersection $\hat{X} \cap \hat{Y}$ 
is the collection of the $3N_2$ (complex) lines intersecting in triplets. Each of these arrangements $L$ of three 
(complex, affine) lines retracts to a triangle, hence it has $H_2(L)=0$, $H_1(L)=\ZZ$ and $H_0(L)=\ZZ$. Thus 
$H_0(\hat{X} \cap \hat{Y}) = H_1(\hat{X} \cap \hat{Y}) = \ZZ^{N_2}$,
and $H_k(\hat{X} \cap \hat{Y}) = 0$ for $k \geq 2$.

Now, applying the Mayer-Vietoris long exact sequence we get a long exact sequence 
\begin{center}
\begin{tikzpicture}[descr/.style={fill=white,inner sep=1.5pt}]
        \matrix (m) [
            matrix of math nodes,
            row sep=1em,
            column sep=2.5em,
            text height=1.5ex, text depth=0.25ex
        ]
        {  0  & H_2(\hat{X} \cap \hat{Y}) = 0  & H_2(\hat{X}) \oplus H_2(\hat{Y}) = 0  & H_2(\cX_3) & \\
            & H_1(\hat{X} \cap \hat{Y}) = \ZZ^{N_2} & H_1(\hat{X}) \oplus H_1(\hat{Y}) = \ZZ_2^{N_1}& H_1(\cX_3) &  \\
            & H_0(\hat{X} \cap \hat{Y}) = \ZZ^{N_2} & H_0(\hat{X}) \oplus H_0(\hat{Y}) = \ZZ^{N_2+1} & H_0(\cX_3) = \ZZ & 0\\
        };

        \path[overlay,->, font=\scriptsize,>=latex]
        (m-1-1) edge (m-1-2)
        (m-1-2) edge (m-1-3)
        (m-1-3) edge (m-1-4)
        (m-1-4) edge[out=355,in=175]  (m-2-2)
        (m-2-2) edge (m-2-3)
        (m-2-3) edge (m-2-4)
        (m-2-4) edge[out=355,in=175]  (m-3-2)
        (m-3-2) edge (m-3-3)
        (m-3-3) edge (m-3-4)
        (m-3-4) edge (m-3-5);
\end{tikzpicture}
\end{center}
Let $\iota: \cY_3 \hookrightarrow \cX_3$ be the inclusion map. Observe that, if we restrict to the stratum $\iota: X \cap Y \hookrightarrow \hat{X} \cap \hat{Y}$, the induced map in homology is an isomorphism since both spaces are homotopically equivalent to $N_2$ copies of $S^1$ and $X \cap Y$ are precisely the generators of the homology of $\hat{X} \cap \hat{Y}$. 
Also, regarding the inclusion $\iota: X\to \hat{X}$, we recall that the parameter $y$ of $\hat{X}$ parametrizes the
eigenvalue $\lambda$ modulo $\lambda\sim -\lambda$, hence the generators of $H_1(X)$ map to generators of $H_1(\hat{X})$
producing an isomorphism.
In this manner, we get the two long exact sequences
	\[
\begin{displaystyle}
   \xymatrix
   {0 \ar[r] & H_2(\cX_3) \ar[r] & H_1(\hat{X} \cap \hat{Y}) = \ZZ^{N_2} \ar[r]  & H_1(\hat{X})=\ZZ_2^{N_2}   \ar[r]& H_1(\cX_3) \ar[r]& 0 \\
   0 \ar[r] & H_2(\cY_3) \ar[u]^{\iota_*} \ar[r] & H_1({X} \cap {Y}) = \ZZ^{N_2} \ar[u]^{\cong} \ar[r] 
   & H_1({X})=\ZZ_2^{N_2}\ar[r]\ar[u]^{\cong}  & 0  \\
      }
\end{displaystyle}   
\]
This proves that $\iota_*$ is an isomorphism, and $H_1(\cX_3)=0$. Therefore $\iota$ is a homology equivalence.
\end{proof}

\begin{cor}\label{cor:X3-simply-connected}
The character variety $\cX_3$ is simply-connected.
\end{cor}

\begin{proof}
Notice that $\pi_1(\hat{X}) = \ZZ_2^{N_1}$ since $\pi_1((\CC\x\CC^*)/D) = \textrm{Coker}\, (\pi_1(D) = H_1(D)  \to \pi_1(\CC\x\CC^*) = H_1(\CC\x\CC^*)) = \ZZ_2$. Similarly $\pi_1(\hat{Y}) = 0$, since each connected component is contractible, and $\pi_1(\hat{X} \cap \hat{Y}) = \ZZ^{N_2}$. Moreover, since $H_1(\cX_3) = 0$, the map $\pi_1(\hat{X} \cap \hat{Y}) = \ZZ^{N_2} \to \pi_1(\hat{X}) = \ZZ_2^{N_1}$ is surjective. At this point, the the proof works verbatim to the one of Corollary \ref{cor:Y3-simply-connected}.
\end{proof}

\begin{cor}
For $n=2$ and odd $m$, the inclusion $\cY_3 \hookrightarrow \cX_3$ is a homotopy equivalence.
\end{cor}

\begin{proof}
This is a standard consequence of Hurewicz and Whitehead theorems. We spell out the details here for completeness. In homology, we have a long exact sequence for $k \geq 0$ for the relative homology
	\[
\begin{displaystyle}
   \xymatrix
   {\ldots \ar[r] & H_{k+1}(\cX_3, \cY_3) \ar[r] & H_k(\cY_3) \ar[r]^{\iota_*} & H_k(\cX_3) \ar[r] & H_k(\cX_3, \cY_3) \ar[r] & \ldots    \\
      }
\end{displaystyle}   
\]
Since $\iota_*$ is an isomorphism for all $k \geq 0$, we get that $H_k(\cX_3, \cY_3) = 0$ for all $k$. By Hurewicz theorem, this implies that the relative homotopy $\pi_k(\cX_3, \cY_3) = 0$ for $k \geq 2$. Recall that the relative version requires $f_\#: \pi_2(\cY_3) \to \pi_2(\cX_3)$ to be surjective, which is automatic by the absolute Hurewicz theorem using that both $\cX_3, \cY_3$ are simply-connected. Now, using the long exact sequence in homotopy
	\[
\begin{displaystyle}
   \xymatrix
   {\ldots \ar[r] & \pi_{k+1}(\cX_3, \cY_3) = 0 \ar[r] & \pi_k(\cY_3) \ar[r]^{\iota_\#}  & \pi_k(\cX_3) \ar[r] & \pi_k(\cX_3, \cY_3) = 0  \ar[r] & \ldots }
\end{displaystyle}   
\]
we get that the induced map in homotopy $\iota_\#$ is an isomorphism for all $k \geq 0$. Thus $\iota$ is a weak homotopy equivalence and a fortiori also a homotopy equivalence by Whitehead theorem.
\end{proof}

Motivated by this result, we predict that the inclusion $\cY_3\hookrightarrow \cX_3$ is also a homotopy equivalence 
for $n,m>2$.


\end{document}